\documentclass{amsart}
\usepackage{amssymb,amsmath,amsthm,amsxtra}
\usepackage{hyperref}
\usepackage[all]{xy}
\usepackage[usenames]{color}

\usepackage{amscd}
\usepackage{amsthm}
\usepackage{amsfonts}
\usepackage{amssymb}
\usepackage{mathrsfs}
\usepackage{url}
\baselineskip 18pt \textwidth 16cm  \theoremstyle{plain}

\newtheorem{thm}{Theorem}[subsection]
\newtheorem{prop}[thm]{Proposition}
\newtheorem{lem}[thm]{Lemma}
\newtheorem{defn}[thm]{Definition}
\newtheorem{notn}[thm]{Notation}
\newtheorem{cor}[thm]{Corollary}

\newtheorem{rem}[thm]{Remark}

\newtheorem*{theorem*}{Theorem}
\newtheorem*{remark*}{Remark}
\newtheorem{lemma}[thm]{Lemma}

\newtheorem{remark}[thm]{Remark}

\newtheorem{theorem}[thm]{Theorem}
\newtheorem{definition}[thm]{Definition}
\newtheorem{notation}[thm]{Notation}

\newtheorem{corollary}[thm]{Corollary}

\newtheorem*{conjecture*}{Conjecture}

\newtheorem{introtheorem}{Theorem}

\newtheorem{introcorollary}[introtheorem]{Corollary}

\newcommand{\al}{\alpha}
\newcommand{\be}{\beta}

\newcommand{\de}{\delta}
\newcommand{\ep}{\epsilon}

\newcommand{\la}{\lambda}

\newcommand{\om}

\newcommand{\De}{\Delta}

\newcommand{\cA}{\mathcal A}
\newcommand{\cB}{\mathcal B}
\newcommand{\cC}{\mathcal C}
\newcommand{\cD}{\mathcal D}
\newcommand{\cE}{\mathcal E}
\newcommand{\cF}{\mathcal F}
\newcommand{\cG}{\mathcal G}
\newcommand{\cH}{\mathcal H}

\newcommand{\cK}{\mathcal K}

\newcommand{\cM}{\mathcal M}

\newcommand{\cO}{\mathcal O}
\newcommand{\cP}{\mathcal P}

\newcommand{\bA}{\mathbb A}

\newcommand{\bC}{\mathbb C}

\newcommand{\bF}{\mathbb F}

\newcommand{\bN}{\mathbb N}

\newcommand{\bQ}{\mathbb Q}
\newcommand{\bR}{\mathbb R}

\newcommand{\bZ}{\mathbb Z}

\newcommand{\fb}{\mathfrak b}

\newcommand{\fX}{\mathfrak X}

\newcommand{\bmat}{\left( \begin{matrix}}
\newcommand{\emat}{\end{matrix}\right)}

\newcommand{\lbl}[1]{\label{#1}}
\newcommand{\nir}[1]{{#1}}

\newcommand{\Dima}[1]{{#1}}
\newcommand{\Rami}[1]{{#1}}
\newcommand{\RGRCor}[1]{{{#1}}}
\newcommand{\DimaGRCor}[1]{{{#1}}}

\newcommand{\marrow}{\marginpar{$\longleftarrow$}}
\newcommand{\Removed}[1]{}

\DeclareMathOperator{\GL}{GL}

\DeclareMathOperator{\M}{Mat}
\DeclareMathOperator{\Stab}{Stab}
\DeclareMathOperator{\Hom}{Hom}
\DeclareMathOperator{\Spec}{Spec}
\DeclareMathOperator{\Img}{Im}

\newcommand{\val}{val}
\newcommand{\res}{res}
\newcommand{\bs}{\backslash}
\newcommand{\Sc}{{\mathcal S}}
\newcommand{\Span}{{\operatorname{Span}}}
\newcommand{\oH}{\operatorname{H}}
\newcommand{\charac}{\operatorname{char}}

\begin{document}

\author{Avraham Aizenbud}
\address{Avraham Aizenbud, Faculty of Mathematics and Computer Science, The Weizmann Institute of Science POB 26, Rehovot 76100, ISRAEL.}
\email{aizner@yahoo.com}
\urladdr{\url{http://www.wisdom.weizmann.ac.il/~aizenr/}}

\author{Nir Avni}
\address{Nir Avni, Department of Mathematics, Harvard University, One Oxford Street Cambridge MA 02138 USA.}
\email{avni.nir@gmail.com}
\urladdr{\url{http://www.math.harvard.edu/~nir}}

\author{Dmitry Gourevitch}
\address{Dmitry Gourevitch, School of Mathematics,
Institute for Advanced Study,
Einstein Drive, Princeton, NJ 08540 USA}
\email{guredim@yahoo.com}
\urladdr{\url{http://www.math.ias.edu/~dimagur/}}

\title{Spherical pairs over close local fields}
\date{\today}

\maketitle
\begin{abstract}
 Extending results of \cite{Kaz} to the relative case, we relate harmonic analysis over some spherical spaces $G(F)/H(F)$, where $F$ is a field of positive characteristic, to harmonic analysis over the spherical spaces $G(E)/H(E)$, where $E$ is a suitably chosen field of characteristic 0. 

\Removed{One of the ingredients of the proof is a condition for finite generation of the space of $K$-invariant compactly supported functions on $G(E)/H(E)$ as a module over the Hecke algebra. }
We apply our results to show that the pair $(\GL_{n+1}(F),\GL_n(F))$ is a strong Gelfand pair for
 all
local fields \Dima{of arbitrary characteristic}, and that the pair $(\GL_{n+k}(F),\GL_n(F)\times\GL_k(F))$ is a Gelfand pair for
local fields
\Dima{of any characteristic different from 2}.
\DimaGRCor{ We also give a criterion for finite generation of the space of $K$-invariant compactly
 supported functions on $G(E)/H(E)$ as a module over the Hecke algebra}.

\end{abstract}

 \tableofcontents

\setcounter{section}{-1}

\section{Introduction}

\marrow Local fields of positive characteristic can be approximated by local fields of characteristic zero. If $F$ and $E$ are local fields, we say that they are $m$-close if $O_F/\cP_F^m \cong O_E/\cP_E^m$, where $O_F, O_E$ are the rings of integers of $F$ and
$E$, and $\cP_F,\cP_E$ are their maximal ideals. \Dima{For example, $F_p((t))$ is $m$-close to $\bQ_p(\sqrt[m]{p})$.}
\Dima{More generally, for any local field $F$ of positive characteristic $p$ and any $m$ there exists a (sufficiently ramified) extension of $\bQ_p$ that is $m$-close to $F$.}

Let $G$ be a reductive group defined over $\bZ$. For any local
field $F$ and conductor $\ell \in \bZ_{\geq 0}$, the Hecke algebra
$\cH_{\ell}(G(F))$ is finitely generated and finitely presented. Based
on this fact, Kazhdan showed in \cite{Kaz} that for any $\ell$ there
exists $m \geq \ell$ such that the algebras $\cH_{\ell}(G(F))$ and $
\cH_{\ell}(G(E))$ are isomorphic for any $m$-close fields $F$ and
$E$. This allows one to transfer certain results in representation theory of reductive groups from
local fields of zero characteristic to local fields of positive
characteristic.

In this paper we investigate a relative version of this technique.
\Dima{Let $G$ be a reductive group and $H$ be a spherical subgroup.
Suppose for simplicity that both are defined over $\bZ$.}

In the first part of the paper
we consider the space $\Sc(G(F)/H(F))^{\Dima{K}}$ of \Dima{compactly supported} 
functions
on $G(F)/H(F)$ which are invariant with respect to \Dima{a compact open} 
subgroup \Dima{$K$}. We prove under certain assumption on the pair
$(G,H)$ that this space is finitely generated (and hence finitely
presented) over the Hecke algebra $\cH_{\Dima{K}}(G(F))$.

\Dima{\begin{introtheorem}[see Theorem \ref{FinGen}] \lbl{IntFinGen}
Let $F$ be a (non-Archimedean) local field. Let $G$ be a
reductive group  and $H<G$ be an algebraic subgroup both defined over $F$. Suppose that
for any parabolic subgroup $P \subset G$, there is a finite number
of double cosets $P(F) \setminus G(F) / H(F)$. Suppose also that
for any irreducible smooth representation $\rho$ of $G(F)$ we have
\begin{equation} \lbl{IntFinMult}
\dim\Hom_{H(F)}(\rho|_{H(F)}, \bC) < \infty .
\end{equation}
Then for any \Dima{compact open} 
subgroup $K<G(F)$, the space
$\Sc(G(F)/H(F))^{K}$ is a finitely generated module over the
Hecke algebra $\mathcal{H}_{K}(G(F))$.
\end{introtheorem}}


Assumption (\ref{IntFinMult}) is rather weak in light of the
results of \cite{Del,SV}. In particular, it holds for all
symmetric pairs over fields of characteristic different from 2.
One can easily show that the converse is also true. Namely, that if
$\Sc(G(F)/H(F))^{K}$ is a finitely generated module over the
Hecke algebra $\mathcal{H}_K(G(F))$ for any \Dima{compact open} 
subgroup
$K<G(F)$, then (\ref{IntFinMult}) holds.

\begin{remark*}
Theorem \ref{IntFinGen} implies that, if $\dim\Hom_{H(F)}(\rho|_{H(F)}, \bC)$ is finite, then it is bounded on every Bernstein component.
\end{remark*}

In the second part of the paper we introduce the notion of a
uniform spherical pair and prove for them the following analog of
Kazhdan's theorem.

\begin{introtheorem} \lbl{thm:ModIso}[See Theorem \ref{thm:phiModIso}]
\Dima{Let $H <G$ be reductive groups defined over
$\bZ$. Suppose that the pair $(G,H)$ is uniform spherical.} 
\Removed{ and the module
$\Sc(G(F)/H(F))^{K_{\ell}(F)}$ is finitely generated over the Hecke
algebra $\cH_{\ell}(G(F))$ for any $F$ and $l$.} 

Then for any $l$ there
exists $n$ such that for any $n$-close local fields $F$ and $E$,
the module $\Sc(G(F)/H(F))^{K_{\ell}(F)}$ over the algebra $\cH_{\ell}(G(F))$
is isomorphic to the module $\Sc(G(E)/H(E))^{K_{\ell}(E)}$ over the
algebra $\cH_{\ell}(G(E))$\nir{, where we identify $\cH_{\ell}(G(F))$ and $\cH_{\ell}(G(E))$ using Kazhdan's isomorphism.}
\end{introtheorem}
\Dima{In fact, we prove a more general theorem, see \S \ref{sec:UniSPairs}.}
\Removed{ Together with Theorem \ref{IntFinGen} }
This implies the following
corollary.
\begin{introcorollary}\lbl{cor:GelGel}
Let $(G,H)$ be a uniform spherical pair of reductive groups
defined over $\bZ$. Suppose that

\begin{itemize}
\item For any local field $F$, and any parabolic subgroup $P \subset G$, there is a finite number
of double cosets $P(F) \setminus G(F) / H(F)$.
\Removed{ \item For any local field $F$ and any irreducible smooth representation $\rho$ of $G(F)$ we have
 $$\dim\Hom_{H(F)}(\rho|_{H(F)}, \bC) < \infty .$$ }
\item For any local field $F$ of characteristic zero the pair $(G(F),H(F))$ is a Gelfand pair, i.e. for
any irreducible smooth representation $\rho$ of $G(F)$ we have
$$\dim\Hom_{H(F)}(\rho|_{H(F)}, \bC) \leq 1 .$$
\end{itemize}
Then for any local field $F$ the pair $(G(F),H(F))$ is a  Gelfand
pair.
\end{introcorollary}
\Dima{In fact, we prove a more general theorem, see \S \ref{sec:UniSPairs}.}

\begin{remark*}
In a similar way one can deduce an analogous corollary for cuspidal representations. Namely, suppose that the first two conditions of the last corollary hold and the third condition holds for all cuspidal representations $\rho$. Then for any local field $F$ the pair $(G(F),H(F))$ is a cuspidal Gelfand pair: for any irreducible smooth cuspidal representation $\rho$ of $G(F)$ we have
$$\dim\Hom_{H(F)}(\rho|_{H(F)}, \bC) \leq 1 .$$
\end{remark*}

\DimaGRCor{
\begin{remark*}
Originally, we included in the formulation of Theorem  \ref{thm:ModIso} an extra condition: we demanded that the module
$\Sc(G(F)/H(F))^{K_{\ell}(F)}$ is finitely generated over the Hecke
algebra $\cH_{\ell}(G(F))$ for any $F$ and $l$. This was our original motivation for Theorem \ref{IntFinGen}.
Later we realized that this condition just follows from the definition of uniform spherical pair. However, we think that  Theorem  \ref{IntFinGen}  and the technique we use in its proof have importance of their own.
\end{remark*}}

In the last part of the paper we apply our technique to show
that $(\GL_{n+1},\GL_n)$ is a strong Gelfand pair over any local
field and $(\GL_{n+k},\GL_n \times \GL_k)$ is a Gelfand pair over
any local field of odd characteristic.

\begin{introtheorem} \lbl{thm:Mult1}
Let $F$ be any local field. Then $(\GL_{n+1}(F),\GL_n(F))$ is a
strong Gelfand pair, i.e. for any irreducible smooth
representations $\pi$ of $\GL_{n+1}(F)$ and $\tau$ of $\GL_{n}(F)$
we have
$$\dim \Hom_{\GL_{n}(F)} (\pi, \tau) \leq 1.$$
\end{introtheorem}

\begin{introtheorem} \lbl{thm:UniLinPer}
Let $F$ be any local field. Suppose that $\charac F \neq 2$. Then
$(\GL_{n+k}\Dima{(F)},\GL_n\Dima{(F)} \times \GL_k\Dima{(F)})$ is a Gelfand pair.
\end{introtheorem}

We deduce these theorems from the zero characteristic case, which
was proven in \cite{AGRS} and \cite{JR} respectively. The proofs
in \cite{AGRS} and \cite{JR} cannot be directly adapted to the
case of positive characteristic since they rely on Jordan
decomposition which is problematic in positive characteristic, local fields of positive characteristic being non-perfect.

\begin{remark*}
In \cite{AGS}, a special case of  Theorem \ref{thm:Mult1} was
proven for all local fields; namely the case when $\tau$ is
one-dimensional.
\end{remark*}

\begin{remark*}
In \cite{AG_AMOT} and (independently) in \cite{SZ}, an analog of
Theorem \ref{thm:Mult1} was proven for Archimedean local fields.
In \cite{AG_HC}, an analog of Theorem \ref{thm:UniLinPer} was
proven for Archimedean local fields.
\end{remark*}

\subsection{Structure of the paper}$ $

\Dima{In Section \ref{sec:PrelNot} we introduce notation and give some general preliminaries.}\Rami{\\}

In Section \ref{sec:FinGen} we prove Theorem \ref{IntFinGen}.

In Subsection \ref{subsec:prel} we collect a few general facts for the proof. One is a criterion, due to Bernstein, for finite generation of the space of $K$-invariant vectors in a representation of a reductive group $G$; the other facts concern homologies of $l$-groups. In Subsection \ref{subsec:desc.cusp} we prove the main inductive step in the proof of Theorem \ref{IntFinGen}, and in Subsection \ref{SecPfFinGen} we prove Theorem \ref{IntFinGen}. Subsection \ref{subsec:homologies} is devoted to the proofs of some facts about the homologies of $l$-groups.\\

In Section \ref{sec:UniSPairs} we prove Theorem \ref{thm:ModIso} and derive Corollary \ref{cor:GelGel}.

In Subsection \ref{subsec:UniSpairs} we introduce the notion of uniform spherical pair. In Subsection \ref{subsec:CloseLocalFields} we prove the theorem and the corollary.\\

We apply our results in Section \ref{sec:ap}. In Subsection \ref{subsec:JR} we prove that the pair $(\GL_{n+k}, \GL_n \times\GL_k )$ satisfies the assumptions of Corollary \ref{cor:GelGel} over fields of characteristic different from 2. In Subsections \ref{subsec:GL} and \ref{subsec:PfGood} we prove that the pair
$(\GL_{n+1}\times\GL_n , \Delta \GL_n)$ satisfies the assumptions of Corollary \ref{cor:GelGel}.
These facts imply Theorems  \ref{thm:Mult1} and \ref{thm:UniLinPer}.
\subsection{Acknowledgments}$ $

We thank {\bf Joseph Bernstein} for directing us to the paper
\cite{Kaz}.

We also thank {\bf Vladimir Berkovich}, {\bf Joseph Bernstein}, {\bf Pierre Deligne},
{\bf Patrick Delorme}, {\bf Jochen Heinloth}, \RGRCor{{\bf Anthony Joseph,}}
{\bf David Kazhdan} ,{\bf Yiannis
Sakelaridis}, and {\bf Eitan Sayag} for fruitful discussions and the referee for many useful remarks.

A.A. was supported by a BSF grant, a GIF grant, an ISF Center
of excellency grant and ISF grant No. 583/09.
N.A. was supported by NSF grant DMS-0901638.
D.G. was supported by NSF grant DMS-0635607. Any opinions, findings and conclusions or recommendations expressed in this material are those of the authors and do not necessarily reflect the views of the National Science Foundation.

\section{Preliminaries and notation} \lbl{sec:PrelNot}

\begin{defn}
 A local field is a locally compact complete non-discrete topological field. In this paper we will consider only non-Archimedean local fields. All such fields have discrete valuations.
\end{defn}

\begin{remark}
 Any local field of characteristic zero and residue  characteristic $p$ is a finite extension of the field $\bQ_p$ of $p$-adic numbers and any local field of characteristic $p$ is a finite extension of the field $\bF_p((t))$ of formal Laurent series over the field with $p$ elements.
\end{remark}

\begin{notn}
For a local field $F$ we denote by $\val_F$ its valuation, by $O_F$ the ring of integers and by $\cP_F$ its unique maximal ideal.
For an algebraic group $G$ defined over $O_F$ we denote by $K_{\ell}(G,F)$ the kernel of the (surjective) morphism $G(O_F)\to G(O_F/\cP_F^{\ell})$. If $\ell>0$ then we call  $K_{\ell}(G,F)$ the $\ell$-th congruence subgroup.
\end{notn}
\Rami{
We will use the terminology of $l$-spaces and $l$-groups introduced in \cite{BZ}. An $l$-space is a locally compact second countable totally disconnected topological space, an $l$-group is a $l$-space with a continuous group structure. For \Dima{further background} on $l$-spaces, $l$-groups and their representations we refer the reader to \cite{BZ}.}

\begin{notn}
Let $G$ be an $l$-group.
Denote by
 $\cM(G)$ the category of smooth \nir{complex} representations of $G$.

Define the functor of coinvariants $CI_G:\cM(G) \to Vect$ by
 $$CI_G(V):= V/(\Span\{v-gv\, | \, v\in V, \, g\in G\}).$$
Sometimes we will also denote $V_G:=CI_G(V)$.
\end{notn}

\begin{notn}
For an $l$-space $X$ we denote by
 $\Sc(X)$ the space of locally constant compactly supported complex valued functions on $X$
\Rami{. If X is an analytic variety over a non-Archimedean local field, we denote}
 by  $\cM(X)$ the space of locally constant compactly supported measures on $X$.
\end{notn}

\begin{notn}
For an $l$-group $G$ and an open compact subgroup $K$ we denote by
 $\cH(G,K)$ or $\cH_K(G)$ the Hecke algebra of $G$ w.r.t. $K$, i.e. the algebra of compactly supported measures on $G$ that are invariant w.r.t. both left and right multiplication by $K$.

For a local field $F$ and a reductive group $G$ defined over $O_F$ we will also denote $\cH_{\ell}(G(F)):=\cH_{K_{\ell}(G)}(G(F))$.
\end{notn}

\nir{\begin{notn}
By a reductive group over a ring $R$, we mean a smooth group scheme over $\Spec(R)$ all of whose geometric fibers are reductive and connected.
\end{notn}}

\section{Finite Generation of Hecke Modules} \lbl{sec:FinGen}

The goal of this section is to prove Theorem \ref{IntFinGen}.

In this section $F$ is a fixed (non-Archimedean) local field of
arbitrary characteristic. All the algebraic groups and algebraic
varieties that we consider in this section are defined over $F$. \nir{In particular, reductive means reductive over $F$}.

\nir{For the reader's convenience, we now give an overview of the argument. In Lemma \ref{VKFinGen} we present a criterion, due to Bernstein, for the finite generation of spaces of $K$-invariants. The proof of the criterion uses the theory of Bernstein Center. This condition is given in terms of all parabolic subgroups of $G$. We directly prove this condition when the parabolic is $G$ (this is Step 1 in the proof of  Theorem \ref{IntFinGen}). The case of general parabolic is reduced to the case where the parabolic is $G$. For this, the main step is to show that the assumptions of the theorem imply similar assumptions for the Levi components of the parabolic subgroups of $G$. This is proved in Lemma \ref{FinMultCuspDesc} by stratifying the space $G(F)/P(F)$ according to the $H(F)$-orbits inside it.
}
\Dima{In the proof of this lemma we use two homological tools: Lemma \ref{FinDimH1H0} that which gives a criterion for finite dimensionality of the first homology of a representation and
Lemma \ref{LemShap} which connects the homologies of a representation and of its induction.}

\subsection{Preliminaries} \lbl{subsec:prel}

\Dima{
\begin{notn}
 For $l$-groups $H<G$ we denote by $ind_H^G: \cM(H) \to \cM(G)$ the compactly supported induction functor and by $Ind_H^G: \cM(H) \to \cM(G)$ the full  induction functor. 
\end{notn}
}

\begin{defn}
\Dima{Let $G$ be a reductive group, let $P<G$ be a parabolic subgroup with unipotent radical $U$, and let $M:=P/U$.
Such $M$ is called a Levi subquotient of $G$.
Note that every representation of $M(F)$ can be considered as a  representation of $P(F)$ using the quotient morphism $P \twoheadrightarrow M$.
Define:
\begin{enumerate}
 \item The Jacquet functor $r_{GM}:\cM(G(F)) \to \cM(M(F))$ by $r_{GM}(\pi):=(\pi|_{P(F)})_{U(F)}$.
 \item The parabolic induction functor $i_{GM}:\cM(M(F)) \to \cM(G(F))$ by $i_{GM}(\tau):=ind_{P(F)}^{G(F)}(\tau)$. 
\end{enumerate}
Note that 
$i_{GM}$ is right adjoint to $r_{GM}$.
A representation $\pi$ of $G(F)$ is called cuspidal if $r_{GM}(\pi)=0$ for any Levi subquotient $M$ of $G$.}
\end{defn}

\begin{definition}
Let $G$ be an $l$-group. A smooth representation $V$ of $G$ is
called \textbf{compact} if for any $v \in V$ and $\xi \in
\widetilde{V}$ the matrix coefficient function defined by
$m_{v,\xi}(g):= \xi(gv)$ is a compactly supported function on $G$.
\end{definition}

\begin{theorem}[Bernstein-Zelevinsky]\lbl{CompProj}
Let $G$ be an $l$-group. Then any compact representation of $G$ is
a projective object in the category $\cM(G)$.
\end{theorem}

\begin{definition}
Let $G$ be a reductive group. \\
(i) Denote by $G^1$ the preimage in $G(F)$ of the maximal compact
subgroup of $G(F)/[G,G](F)$.\\
(ii) Denote $G_0:=G^1Z(G(F))$.\\
(iii) A complex character of $G(F)$ is called unramified if it is trivial
on $G^1$. We denote the \Dima{set} of all unramified
characters by $\Psi_G$. \Dima{Note that $G(F)/G^1$ is a lattice and therefore we can identify $\Psi_G$ with $(\bC^{\times})^n$. This defines a structure of algebraic variety on $\Psi_G$.}\\
(iv) \Dima{For any smooth representation $\rho$ of $G(F)$ we denote
$\Psi(\rho):= ind_{G^1}^G(\rho|_{G^1})$. Note that $\Psi(\rho) \simeq  \rho \otimes \cO(\Psi_G),$
where $G(F)$ acts only on the first factor, but this action depends on the  second factor.
This identification gives a structure of $\cO(\Psi_G)$-module on $\Psi(\rho)$.}
\end{definition}

\nir{\begin{rem} The definition of unramified characters above is not the standard one, but it is more convenient for our purposes.
\end{rem}}

\begin{theorem}[Harish-Chandra]\lbl{CuspComp}
Let $G$ be a reductive group and $V$ be a cuspidal representation
of $G(F)$. Then $V|_{G^1}$ is a compact representation of $G^1$.
\end{theorem}

\begin{corollary} \lbl{CuspProj}
Let $G$ be a reductive group and $\rho$ be a cuspidal
representation
of $G(F)$. Then\\
(i) $\rho|_{G^1}$ is a projective object in the category
$\cM(G^1)$.\\
(ii) $\Psi(\rho)$ is a projective object in the category
$\cM(G(F))$.
\end{corollary}
\begin{proof}
(i) is clear.\\
(ii) note that $$Hom_G(\Psi(\rho),\pi) \cong
Hom_{G/G_1}(\cO(\Psi_M),Hom_{G^1}(\rho,\pi)),$$
for any representation $\pi$. Therefore the functor $\pi \mapsto
Hom_G(\Psi(\rho),\pi)$ is a composition of two exact functors and
hence is exact.
\end{proof}

\Dima{
\begin{defn}
Let $G$ be a reductive group and $K<G(F)$ be a compact open subgroup.  Denote $$\cM(G,K):= \{V \in \cM(G(F))\, | V \text{ is generated by }V^K\}$$ and $$\cM(G,K)^{\bot}:=
\{V \in \cM(G(F)\, | V^K = 0\}.$$
We call $K$ a splitting subgroup if the category $\cM(G(F))$ is the direct sum of the categories $\cM(G,K)$ and $\cM(G,K)^{\bot}$, and $\cM(G,K) \cong \cM(\cH_K(G))$. \nir{Recall that an abelian category $\mathcal{A}$ is a direct sum of two abelian subcategories $\mathcal{B}$ and $\mathcal{C}$, if every object of $\mathcal{A}$ is isomorphic to a direct sum of an object in $\mathcal{B}$ and an object in $\mathcal{C}$, and, furthermore, that there are no non-trivial
\Dima{ morphisms }
between objects of $\mathcal{B}$ and $\mathcal{C}$.}
\end{defn}
}

We will use the following statements from Bernstein's theory on
the center of the category $\cM(G)$.
\Dima{Let $P<G$ be a parabolic subgroup and $M$ be the reductive quotient of $P$.}

\begin{enumerate}
\item \lbl{1}  \Dima{ The set of splitting subgroups defines a basis at 1 for the topology of $G(F)$.} \nir{If $G$ splits over
\Rami{$O_F$}
then, for any large enough $\ell$, the congruence subgroup $K_\ell(G,F)$ is splitting.}
\item \lbl{2} \Dima{Let $\overline{P}$ denote the parabolic subgroup of $G$ opposite to $P$, and let $\overline{r}_{GM}:\cM(G(F)) \to \cM(M(F))$ denote the Jacquet functor defined using $\overline{P}$.
Then $\overline{r}_{GM}$ is right adjoint to $i_{GM}$. In particular, $i_{GM}$ maps projective objects to projective ones and hence for any irreducible cuspidal
representation $\rho$
of $M(F)$, 
$i_{GM}(\Psi(\rho))$ is a projective object of $\cM(G(F))$.}

\item \lbl{3} Denote by $\cM_{\rho}$ the subcategory of $\cM(G(F))$
generated by $i_{GM}(\Psi(\rho))$. Then $$\cM(G,K) =
\oplus_{(M,\rho) \in B_K} \cM_{\rho},$$ where $B_K$ is some finite
set of pairs consisting of a Levi subquotient of $G$ and its cuspidal
representation. Moreover, for any Levi subquotient $M<G$ and a
cuspidal representation $\rho$ of $M(F)$ such that $\cM_{\rho}
\subset \cM(G,K)$ there exist $(M',\rho ' )\in B_K$ such that
$\cM_{\rho} = \cM_{\rho '}$.
\item \lbl{4} $End(i_{GM}(\Psi(\rho)))$ is finitely generated over
$\cO(\Psi)$ \nir{which is finitely generated over the center of the ring
$End(i_{GM}(\Psi(\rho)))$. The center of the ring $End(i_{GM}(\Psi(\rho)))$ is equal to the center \Dima{$Z(\cM_{\rho})$} of the category $\cM_{\rho}$.}
\end{enumerate}

For statements \ref{1} \Dima{see e.g. \cite[pp. 15-16]{BD} and \cite[\S 2]{HcHvDJ}.}
For
statement \ref{2} see \cite{BerSec} or \cite[Theorem 3]{Bus}. 
For statements \ref{3},\ref{4} see \cite[Proposition 2.10,2.11]{BD}. 


\Dima{We now present} a criterion, due to Bernstein, for finite generation of the space $V^K$, consisting of vectors in a representation $V$ that are invariant with respect to \Dima{a compact open} 
subgroup $K$.
\Rami{
\begin{lemma} \lbl{VKFinGen}
Let $V$ be a smooth representation of $G(F)$. Suppose that for any
parabolic $P<G$ and any irreducible cuspidal representation $\rho$
of $M(F)$ (where $M$ denotes the \Dima{reductive quotient} 
of $P$),
$ \Hom_{G(F)}(i_{GM}(\Psi(\rho)),V)$ is a finitely generated
module over $\cO(\Psi_M)$. Then $V^K$ is a finitely generated
module over  $\Dima{Z(\mathcal{H}_K(G(F)))}$, for any compact open 
subgroup $K<G(F)$.
\end{lemma}
}
\begin{proof}$ $\\
\Rami{
Step 1.  Proof for the case when $K$ is splitting and $V=i_{GM}(\Psi(\rho))$ for some Levi
subquotient $M$ of $G$ and an irreducible cuspidal representation $\rho$
of $M(F).$
\Dima{Let $P$ denote the parabolic subgroup that defines $M$ and $U$ denote its unipotent radical.
Denote $K_M:=K/(U\Dima{(F)}\cap K)<M\Dima{(F)}$.}
If $V^K=0$ there is nothing to prove. Otherwise $\cM_{\rho}$ is a
direct summand of $\cM(G,K)$. Now
$$V^K = \Psi(\rho)^{K_M} = \rho^{K_M} \otimes \cO(\Psi).$$
Hence $V^K$ is finitely generated over $Z(\cM_{\rho})$. Hence
$V^K$ is finitely generated over $Z(\cM(G,K)) = Z(\cH_K(G))$.
}

\Rami{
Step 2. Proof for the case when $K$ is splitting and $V \in \cM_{\rho}$ for some Levi
subquotient $M<G$ and an irreducible cuspidal representation $\rho$
of $M(F)$.
}

\Rami{
Let $$\phi: i_{GM}(\Psi(\rho)) \otimes \Hom(i_{GM}(\Psi(\rho)),V)
\twoheadrightarrow V$$ be the natural epimorphism. We are given that
$\Hom(i_{GM}(\Psi(\rho)),V)$ is finitely generated over
$\cO(\Psi)$. Hence it is finitely generated over $Z(\cM(\rho))$.
Choose some generators $\alpha_1, ..., \alpha_n \in
\Hom(i_{GM}(\Psi(\rho))$. Let $$\psi: i_{GM}(\Psi(\rho))^n
\hookrightarrow i_{GM}(\Psi(\rho)) \otimes
\Hom(i_{GM}(\Psi(\rho)),V)$$ be the corresponding morphism.
$Im(\phi \circ \psi)$ is $Z(\cM(\rho))$-invariant and hence
coincides with $Im(\phi)$. Hence $\phi \circ \psi$ is onto. The
statement now follows from the previous step.
}

\Rami{
Step 3. Proof for the case when $K$ is splitting.
}

\Rami{
Let $W<V$ be the subrepresentation generated by $V^K$. By
definition $W \in \cM(G,K)$ and hence $W = \oplus_{i=1}^n W_i$
where $W_i \in \cM_{\rho_i}$ for some $\rho_i$. The lemma now
follows from the previous step.
}

\Rami{
Step 4. General case
}

\Rami{
Let $K'$ be a splitting subgroup s.t. $K'<K$.
Let $v_1...v_n \in V^{K'}$ be the generators of $V^{K'}$ over
$Z(\mathcal{H}_{K'}(G(F)))$ given by the previous step. Define $w_i:=e_{K}\Dima{v_i} \in V^{K}$ where  $e_{K}\in \mathcal{H}_{K}(G(F))$ is the normalized Haar measure of $K.$ Let us show that $w_i$ generate $V^{K}$ over

$Z(\mathcal{H}_{K}(G(F)))$.  Let $x \in  V^{K}$. We can represent $x$ as a sum $\sum h_i v_i$, where $h_i \in Z(\mathcal{H}_{K'}(G(F)))$. Now $$x=e_{K}x=\sum e_{K}h_i v_i=\sum e_{K} e_{K}h_i v_i=\sum e_{K} h_i e_{K}v_i=\sum e_{K} h_i e_{K} e_{K}v_i= \sum e_{K} h_i e_{K} w_i.$$}
\end{proof}

Finally, in this subsection, we state two facts about homologies of $l$-groups. The proofs and relevant definitions are in Subsection \ref{subsec:homologies}.

\begin{lemma} \lbl{FinDimH1H0}
Let $G$ be an algebraic group and $U$ be its unipotent radical.
Let $\rho$ be an irreducible cuspidal representation of
$(G/U)(F)$. We treat $\rho$ as a representation of $G(F)$ with
trivial action of $U(F)$.

Let $H<G$ be an algebraic subgroup. Suppose that the space of
coinvariants $\rho_{H(F)}$ is finite dimensional. Then $\dim
\oH_1(H(F),\rho)< \infty .$
\end{lemma}

We will also use the following version of Shapiro Lemma.

\begin{lemma} \lbl{LemShap}
Let $G$ be an $l$-group that acts transitively on an $l$-space
$X$. Let $\cF$ be a $G$-equivariant sheaf over $X$. Choose a point
$x \in X$, let $\cF_x$ denote the stalk of $\cF$ at $x$ and $G_x$
denote the stabilizer of $x$. Then
$$\oH_i(G,\cF(X))= \oH_i(G_x,\cF_x).$$
\end{lemma}

\subsection{Descent Of
Finite Multiplicity} \lbl{subsec:desc.cusp}

\begin{definition}
We call a pair $(G,H)$ consisting of a reductive group $G$ and an
algebraic subgroup $H$ \textbf{an $F$-spherical pair} if for any
parabolic subgroup $P \subset G$, there is a finite number of
double cosets in $P(F) \setminus G(F) / H(F)$.
\end{definition}

\begin{remark}
If $char F=0$ \Dima{and $G$ is quasi-split over $F$} then $(G,H)$ is an $F$-spherical pair if and only if
it is a spherical pair of algebraic groups. However, we do not
know whether this is true if $char F>0$.
\end{remark}

\begin{notation}
Let $G$ be a reductive group and $H$ be a subgroup. Let $P<G$ be
a parabolic subgroup and $M$ be its Levi quotient. We denote by
$H_M$ the image of \nir{$H\cap P$} under the projection $P \twoheadrightarrow
M$.
\end{notation}

The following Lemma is the main step in the proof of Theorem \ref{IntFinGen}

\begin{lemma} \lbl{FinMultCuspDesc}
Let $(G,H)$ be an $F$-spherical pair. Let $P<G$ be a parabolic
subgroup and $M$ be its Levi \nir{quotient}.  Then\\
(i) $(M, H_M)$ is also an $F$-spherical pair.\\
(ii) Suppose also that for any smooth irreducible representation
$\rho$ of $G(F)$ we have $$\dim\Hom_{H(F)}(\rho|_{H(F)}, \bC) <
\infty.$$ Then for any irreducible cuspidal representation $\sigma$
of $M(F)$ we have $$\dim\Hom_{H_M(F)}(\sigma|_{H_M(F)}, \bC) <
\infty.$$
\end{lemma}

\begin{remark}
One can show that the converse of (ii) is also true. Namely, if
$\dim\Hom_{H_M(F)}(\sigma|_{H_M(F)}, \bC) < \infty$ for any
irreducible cuspidal representation $\sigma$ of $M(F)$ for any
Levi subquotient $M$ then $\dim\Hom_{H(F)}(\rho|_{H(F)}, \bC) <
\infty$ for any smooth irreducible representation $\rho$ of
$G(F)$. We will not prove this since we will not use this.
\end{remark}

We will need the following lemma.

\begin{lemma} \lbl{LinAlg}
Let $M$ be an l-group and $V$ be a smooth representation of $M$.
Let $0=F^0V \subset ... \subset F^{n-1}V \subset F^nV=V$ be a
finite filtration of $V$ by subrepresentations. Suppose that for
any $i$, either $$\dim (F^iV/F^{i-1}V)_M = \infty$$ or $$\text{both }\dim
(F^iV/F^{i-1}V)_M < \infty \text{ and }\dim \oH_1(M,(F^iV/F^{i-1}V)) <
\infty.$$ Suppose also that $\dim V_M < \infty$. Then $\dim
(F^iV/F^{i-1}V)_M < \infty$ for any $i$.
\end{lemma}
\begin{proof}
\nir{We prove by a decreasing induction on $i$ that $\dim(F^iV)_M<\infty$, and, therefore, $\dim(F^iV/F^{i-1}V)_M<\infty$. Consider the short exact sequence
$$0 \to F^{i-1}V \to F^iV \to F^iV/F^{i-1}V \to 0,$$
and the corresponding long exact sequence
$$...\to \oH_1(M,(F^iV/F^{i-1}V))  \to (F^{i-1}V)_M \to (F^iV)_M \to (F^iV/F^{i-1}V)_M \to 0.$$
In this sequence $\dim \oH_1(M,(F^iV/F^{i-1}V))< \infty$ and $\dim
(F^iV)_M < \infty$, and hence $\dim (F^{i-1}V)_M < \infty$.}
\end{proof}

Now we are ready to prove Lemma \ref{FinMultCuspDesc}.
\begin{proof}[Proof of Lemma \ref{FinMultCuspDesc}]
$ $\\
(i) is trivial.\\
(ii) Let $P<G$ be a parabolic subgroup, $M$ be the Levi quotient
of $P$ and let $\rho$ be a cuspidal representation of $M(F)$. We
know that $\dim (i_{GM}\rho)_{H(F)}< \infty$ and we have to show
that $\dim \rho_{H_M(F)}< \infty$.

Let $\mathcal{I}$ denote the natural $G(F)$-equivariant \nir{locally constant sheaf of complex vector spaces} on
$G(F)/P(F)$ such that $i_{GM}\rho \cong \Sc(G(F)/P(F),
\mathcal{I})$. Let $Y_j$ denote the $H(F)$ orbits on $G(F)/P(F)$.
We know that there exists a natural filtration on $\Sc(G(F)/P(F),
\mathcal{I})|_{H(F)}$ with associated graded components isomorphic
to $\Sc(Y_j,\mathcal{I}_j)$, where $\mathcal{I}_j$ are $H(F)$-
equivariant sheaves on $Y_j$ corresponding to $\mathcal{I}$. For
any $j$ choose a representative $y_j \in Y_j$. Do it in such a way
that there exists $j_0$ such that $y_{j_0}=[1]$. Let
$P_j:=G_{y_j}$ and $M_j$ be its Levi quotient. Note that $P_{j_0} = P$ and
$M_{j_0} = M$. Let $\rho_j$ be the stalk of $\mathcal{I}_j$ at the
point $y_j$. Clearly $\rho_j$ is a cuspidal irreducible
representation of $M_j(F)$ and $\rho_{j_0}=\rho$. By Shapiro Lemma
(Lemma \ref{LemShap}) $$\oH_i(H(F),\Sc(Y_j,\mathcal{I}_j)) \cong
\oH_i((H \cap P_j)(F), \rho_j).$$
By Lemma \ref{FinDimH1H0} either $\dim \oH_0((H \cap P_j)(F),
\rho_j) = \infty$ or both $\dim \oH_0((H \cap P_j)(F), \rho_j) <
\infty$ and $\dim \oH_1((H \cap P_j)(F), \rho_j) < \infty$. Hence
by Lemma \ref{LinAlg} $\dim \oH_0((H \cap P_j)(F), \rho_j) <
\infty$ and hence $\dim \rho_{H_M(F)}< \infty$.
\end{proof}

\subsection{Proof of Theorem \ref{IntFinGen}} \lbl{SecPfFinGen} $ $

In this subsection we prove Theorem \ref{IntFinGen}. Let us remind
its formulation.
\begin{theorem} \lbl{FinGen}
Let $(G,H)$ be an $F$-spherical pair. Suppose that for any
irreducible smooth representation $\rho$ of $G(F)$ we have
\begin{equation} \lbl{FinMult}
\dim\Hom_{H(F)}(\rho|_{H(F)}, \bC) < \infty .
\end{equation}
Then for any \Dima{compact open} 
subgroup $K<G(F)$, $\Sc(G(F)/H(F))^K$ is a
finitely generated module over the Hecke algebra
$\mathcal{H}_K(G(F))$.
\end{theorem}

\begin{remark}
Conjecturally, any $F$-spherical pair satisfies the condition
(\ref{FinMult}). In \cite{Del} and in \cite{SV} this is proven for
wide classes of spherical pairs, which include all symmetric pairs
over fields of characteristic different from 2.
\end{remark}

We will need several lemmas and definitions.

\nir{

\begin{lemma} \lbl{G1}
Let $(G,H)$ be an $F$-spherical pair, and denote $\widetilde{H}=H(F)Z(G(F))\cap G^1$.
Suppose that for any smooth
(respectively cuspidal) irreducible representation $\rho$ of
$G(F)$ we have $\dim\Hom_{H(F)}(\rho|_{H(F)}, \bC) < \infty$. Then
for any smooth (respectively cuspidal) irreducible representation
$\rho$ of $G(F)$ and for every character $\widetilde{\chi}$ of $\widetilde{H}$ whose restriction to $H(F)\cap G^1$ is trivial, we have
$$\dim\Hom_{\widetilde{H}}(\rho|_{\widetilde{H}}, \widetilde{\chi}) < \infty.$$
\end{lemma}
}

\begin{proof} 
\nir{
Let $\rho$ be a smooth (respectively cuspidal) irreducible
representation of $G(F)$, and let $\widetilde{\chi}$ be a character of $\widetilde{H}$ whose restriction to $H(F)\cap G^1$ is trivial.
\[
\Hom_{\widetilde{H}}\left (\rho|_{\widetilde{H}}, \widetilde{\chi} \right ) =
\Hom_{(H(F) Z(G(F)))\cap G_0} \left (\rho|_{(H(F) Z(G(F)))\cap G_0}, Ind_{\widetilde{H}}^{(H(F) Z(G(F)))\cap G_0}\widetilde{\chi}\right) .
\]
Since $$H(F)Z(G(F))\cap G_0=\widetilde{H}Z(G(F))\cap G_0=\widetilde{H}Z(G(F)),$$
the subspace of $Ind_{\widetilde{H}}^{(H(F) Z(G(F)))\cap G_0}\widetilde{\chi}$ that transforms under $Z(G(F))$ according to the central character of $\rho$ is at most one dimensional. If this subspace is trivial, then the lemma is clear. Otherwise, denote it by $\tau$. Since $H(F)\cap G^1$ is normal in $H(F)Z(G(F))$, we get that the restriction of $Ind_{\widetilde{H}}^{(H(F) Z(G(F)))\cap G_0}\widetilde{\chi}$ to $H(F)\cap G^1$ is trivial, and hence that $\tau|_{H(F)\cap G^1}$ is trivial. Hence $\Hom_{\widetilde{H}}\left (\rho|_{\widetilde{H}}, \widetilde{\chi} \right )$ is equal to
\[
\Hom_{(H(F) Z(G(F)))\cap G_0} \left (\rho|_{(H(F) Z(G(F)))\cap G_0}, \tau\right)=\Hom_{H(F)\cap G_0} \left (\rho|_{H(F) \cap G_0}, \tau|_{H(F)\cap G_0}\right)=
\]
\[
=\Hom_{H(F)}\left(\rho|_{H(F)},Ind_{H(F)\cap G_0}^{H(F)}\tau|_{H(F)\cap G_0}\right).
\]
Since $H(F) / H(F)\cap G_0$ is finite and abelian, the representation $Ind_{H(F)\cap G_0}^{H(F)}\tau|_{H(F)\cap G_0}$ is a finite direct sum of characters of $H(F)$, the restrictions of all to $H(F)\cap G^1$ are trivial. Any character $\theta$ of $H(F)$ whose restriction to $H(F)\cap G^1$ is trivial can be extended to a character of $G(F)$, because $H(F)/(H(F)\cap G^1)$ is a sub-lattice of $G(F)/G^1$. Denoting the extension by $\Theta$, we get that
\[
\Hom_{H(F)}\left(\rho|_{H(F)},\theta\right)=\Hom_{H(F)}\left((\rho\otimes\Theta^{-1})|_{H(F)},\bC\right),
\]
but $\rho\otimes\Theta^{-1}$ is again smooth (respectively cuspidal) irreducible representation of $G(F)$, so this last space is finite-dimensional.
}

\end{proof}

\Rami{
 \begin{lemma} \lbl{CA}
 Let $A$ be a commutative unital Noetherian algebra without zero divisors and let $K$ be its field of fractions. Let $K^\bN$ be the space of all sequences of elements of $K$. Let $V$ be a finite dimensional subspace of $K^\bN$ and let $M:=V \cap A^\bN$. Then $M$ is finitely generated.
\end{lemma}

\nir{
\begin{proof} Since $A$
\Rami{
does not have zero divisors}, $M$ injects into $K^\bN$. There is a number $n$ such that the projection of $V$ to $K^{\{1,\ldots n\}}$ is injective. Therefore, $M$ injects into $A^{\{1,\ldots n\}}$, and, since $A$ is Noetherian, $M$ is finitely generated.
\end{proof}
}
 \nir{
 \begin{lemma} \lbl{fg}
Let $M$ be an $l$-group, let $L\subset M$ be a closed subgroup, and let $L' \subset L$ be an open normal subgroup of $L$ such that $L/L'$  is a lattice. Let $\rho$ be a smooth representation of $M$ of countable dimension. Suppose that for any character $\chi$ of $L$ whose restriction to $L'$ is trivial we have $$\dim\Hom_{L}(\rho|_{L}, \chi) < \infty.$$
Consider $\Hom_{L'}(\rho, \Sc(L/L'))$ as a representation of $L$, where $L$ acts by $((hf)(x))([y])=(f(x))([y h])$. Then this representation is finitely generated. \end{lemma}

\begin{proof} By assumption, the action of $L$ on $\Hom_{L'}(\rho,\Sc(L/L'))$ factors through $L/L'$. Since $L/L'$ is discrete, $\Sc(L/L')$ is the group algebra $\bC[L/L']$. We want to show that $\Hom_{L'}(\rho,\bC[L/L'])$ is a finitely generated module over $\bC[L/L']$.

Let $\bC(L/L')$ be the fraction field of $\bC[L/L']$. Choosing a countable basis for the vector space of $\rho$, we can identify any $\bC$-linear map from $\rho$ to $\bC[L/L']$ with an element of $\bC[L/L']^\bN$. Moreover, the condition that the map intertwines the action of $L/L'$ translates into a collection of linear equations that the tuple in $\bC[L/L']^\bN$ should satisfy. Hence, $\Hom_{L'}(\rho,\bC[L/L'])$ is the intersection of the $\bC(L/L')$-vector space $\Hom_{L'}(\rho,\bC(L/L'))$ and $\bC[L/L']^\bN$. By Lemma \ref{CA}, it suffices to prove that $\Hom_{L'}(\rho,\bC(L/L'))$ is finite dimensional over $\bC(L/L')$.

Since 
\Dima{$L'$ is separable, and $\rho$ is smooth and of countable dimension,}
there are only countably many linear equations defining $\Hom_{L'}(\rho,\bC(L/L'))$; denote them by $\phi_1,\phi_2,\ldots\in\left(\bC(L/L')^\bN\right)^*$. Choose a countable subfield $K\subset\bC$ that contains all the coefficients of the elements of $\bC(L/L')$ that appear in any of the $\phi_i$'s. If we define $W$ as the
\Rami{
$K(L/L')$}-linear subspace of
\Rami{
$K(L/L')^\bN$}
\Dima{defined} by the $\phi_i$'s, then $\Hom_{L'}(\rho,\bC(L/L'))=W\otimes_{K(L/L')} \bC(L/L')$, so $\dim_{\bC(L/L')}\Hom_{L'}(\rho,\bC(L/L'))=\dim_{K(L/L')}W$.

Since $L/L'$ is a lattice generated by, say, $g_1,\ldots,g_n$, we get that $K(L/L')=K(t_1^{\pm 1},\ldots,t_n^{\pm 1})$ \Rami{$=K(t_1,\ldots,t_n)$}. Choosing elements $\pi_1,\ldots,\pi_n\in\bC$ such that $tr.deg_K(K(\pi_1,\ldots,\pi_n))=n$, we get an injection $\iota$ of
$K(L/L')$ into $\bC$. As before, we get that if we denote the $\bC$-vector subspace of $\bC^\bN$ cut by the equations $\iota(\phi_i)$ by $U$, then $\dim_{K(L/L')}W=\dim_{\bC}U$. However, $U$ is isomorphic to $\Hom_{L'}(\rho,\chi)$, where $\chi$ is the character \Rami{of $L/L'$} such that $\chi(g_i)=\pi_i$. By assumption, this last vector space is finite dimensional.
 \end{proof}}
}
Now we are ready to prove Theorem \ref{FinGen}.
\begin{proof}[Proof of Theorem \ref{FinGen}]
By Lemma \ref{VKFinGen} it is enough to show that for any
parabolic $P<G$ and any irreducible cuspidal representation $\rho$
of $M(F)$ (where $M$ denotes the Levi quotient of $P$), 
$ \Hom(i_{GM}(\Psi(\rho)),\Sc(G(F)/H(F)))$ is a finitely generated
module over $\cO(\Psi_M)$.

\Rami{
%
%
Step 1. Proof for the case $P=G$.\\
We have
$$\Hom_{G(F)}(i_{GM}(\Psi(\rho)),\Sc(G(F)/H(F)))=
\Hom_{G(F)}(\Psi(\rho),\Sc(G(F)/H(F))) =
\Hom_{G^1}(\rho,\Sc(G(F)/H(F))).$$
Here we consider the space $\Hom_{G^1}(\rho,\Sc(G(F)/H(F)))$ with

the natural action of $G$. Note that $G^1$ acts trivially and
hence this action gives rise to an action of $G/G^1$, which gives
the $\cO(\Psi_G)$ - module structure.

Now consider the subspace $$V:=\Hom_{G^1}(\rho,\Sc(G^1/(H(F) \cap
G^1))) \subset \Hom_{G^1}(\rho,\Sc(G(F)/H(F))).$$  It generates
$\Hom_{G^1}(\rho,\Sc(G(F)/H(F)))$ as a representation of $G(F)$,
and therefore also as an $\cO(\Psi_G)$ - module. Note that $V$ is $H(F)$ invariant. Therefore it is enough to show that $V$ is finitely generated over $H(F).$
Denote $H':=H(F) \cap
G^1$ and $H'':= (H(F) Z(G(F)))\cap G^1$.
Note that $$\Sc(G^1/H') \cong ind_{H''}^{G^1}(\Sc(H''/H')) \subset Ind_{H''}^{G^1}(\Sc(H''/H')).$$ Therefore $V$ is canonically embedded into $\Hom_{H''}(\rho,\Sc(H''/H'))$. The action of $H$ on $V$ is naturally extended to an action $\Pi$ on $\Hom_{H''}(\rho,\Sc(H''/H'))$ by $$((\Pi(h)(f))(v))([k])=f(h^{-1}v)([h^{-1}kh]).$$ Let $\Xi$ be the action of $H''$ on $\Hom_{H''}(\rho,\Sc(H''/H'))$ as described in Lemma \Dima{\ref{fg}}, i.e. $$((\Xi(h)(f))(v))([k])=f(v)([kh]).$$ By Lemmas \Dima{\ref{fg}} and \Dima{\ref{G1}} it is enough to show that for any $h \in H''$ there exist an $h' \in H$ and a scalar $\alpha$ s.t. $$\Xi(h)=\alpha \Pi(h').$$ In order to show this let us decompose $h$ to a product $h=zh'$ where $h' \in H$ and $z\in Z(G(F))$. Now
\begin{multline*}
((\Xi(h)(f))(v))([k])=f(v)([kh])=f(h^{-1}v)([h^{-1}kh])=f(h^{'-1}z^{-1}v)([h^{'-1}kh'])=\\=
\alpha f(h'^{-1}v)([h^{'-1}kh'])= \alpha((\Pi(h')(f))(v))([k]),
\end{multline*}
where $\alpha$ is the scalar with which
$z^{-1}$ acts on $\rho$.}
%

\Rami{Step 2}. Proof in the general case.

\nir{
\begin{multline*}
\Hom_{G(F)}(i_{GM}(\Psi(\rho)),\Sc(G(F)/H(F)))=
\Hom_{M(F)}(\Psi(\rho),\overline{r}_{MG}(\Sc(G(F)/H(F)))) = \\ =
\Hom_{M(F)}(\Psi(\rho),((\Sc(G(F)/H(F)))|_{\overline{P}(F)})_{\overline{U}(F)}),
\end{multline*}
where $\overline{U}$ is the unipotent radical of $\overline{P}$, the parabolic opposite to $P$. Let $\{Y_i\}_{i=1}^n$
be the orbits of $\overline{P}(F)$ on $G(F)/H(F)$. We know that there exists
a filtration on $(\Sc(G(F)/H(F)))|_{\overline{P}(F)}$ such that the
associated graded components are isomorphic to $\Sc(Y_i)$.
Consider the corresponding filtration on
$((\Sc(G(F)/H(F)))|_{\overline{P}(F)})_{\overline{U}(F)}$. Let $V_i$ be the associated
graded components of this filtration. We have a natural surjection
$\Sc(Y_i)_{\overline{U}} \twoheadrightarrow V_i$. In order to prove that
$\Hom_{\Dima{M}(F)}(\Psi(\rho),((\Sc(G(F)/H(F)))|_{\overline{P}(F)})_{\overline{U}(F)})$ is
finitely generated it is enough to prove that
$\Hom_{M(F)}(\Psi(\rho),V_i)$ is finitely generated. Since
$\Psi(\rho)$ is a projective object of $\cM(M(F))$ (by Corollary
\ref{CuspProj}), it is enough to show that
$\Hom_{M(F)}(\Psi(\rho),\Sc(Y_i)_{\overline{U}(F)})$ is finitely generated. Denote
$Z_i := \overline{U}(F) \setminus Y_i $. It is easy to see that $Z_i \cong
M(F)/ ((H_i)_M(F))$, where $H_i$ is some conjugation of $H$.
}
Now the assertion follows from the previous step using Lemma
\ref{FinMultCuspDesc}.
\end{proof}

\subsection{Homologies of $l$-groups}\lbl{subsec:homologies}$ $

The goal of this subsection is to prove Lemma \ref{FinDimH1H0} and
Lemma \ref{LemShap}.

We start with some generalities on abelian categories.
\begin{definition}
Let $\cC$ be an abelian category. We call a family of objects $\cA
\subset Ob(\cC)$ \textbf{generating} family if for any object $X
\in Ob(\cC)$ there exists an object $Y \in \cA$ and an epimorphism
$Y \twoheadrightarrow X$.
\end{definition}

\begin{definition}
Let $\cC$ and $\cD$ be abelian categories and $\cF: \cC \to \cD$
be a right-exact additive functor. A family of objects $\cA
\subset Ob(\cC)$ is called \textbf{$\cF$-adapted} if it is
generating
\Rami{, closed under direct sums
and for any exact sequence $0 \to A_1 \to A_2 \to ...
$ with $A_i \in \cA$, the sequence $0 \to \cF(A_1) \to \cF(A_2) \to...$ is also exact.
}


\Dima{For example,
\Rami{a generating,
closed under direct sums system consisting of}
projective  objects
is $\cF$-adapted for any \Dima{right-exact} 
functor $\cF$.}
\Rami{For an $\l$-group $G$ the system of objects consisting of direct sums of copies of $\Sc(G)$ is an example of such system.}
\end{definition}
The following results are well-known.
\begin{theorem}
Let $\cC$ and $\cD$ be abelian categories and $\cF: \cC \to \cD$
be a right-exact additive functor. Suppose that there exists an
$\cF$-adapted family $\cA \subset Ob(\cC)$. Then $\cF$ has derived
functors.
\end{theorem}

\begin{lemma} \lbl{HomLeib}
Let $\cA$, $\cB$ and $\cC$ be abelian categories. Let $\cF:\cA \to
\cB$ and $\cG:\cB \to \cC$ be right-exact additive functors.
Suppose that both $\cF$ and $\cG$ have derived functors.

(i) Suppose that $\cF$ is exact. Suppose also that there exists a
class $\cE \subset Ob(\cA)$ which is $\cG \circ \cF$-adapted and
such that $\cF(X)$ is $\cG$-acyclic for any $X \in \cE$. Then \nir{the functors $L^i(\cG \circ \cF)$ and $L^i\cG \circ \cF$ are isomorphic.}

(ii) Suppose that there exists a class $\cE \subset Ob(\cA)$ which
is $\cG \circ \cF$-adapted and $\cF$-adapted and such that
$\cF(X)$ is $\cG$-acyclic for any $X \in \cE$. Let $Y \in \cA$ be
an $\cF$-acyclic object. Then \nir{$L^i(\cG \circ \cF)(Y)$ is (naturally) isomorphic to $L^i\cG( \cF(Y))$.}

(iii) Suppose that $\cG$ is exact. Suppose that there exists a
class $\cE \subset Ob(\cA)$ which is $\cG \circ \cF$-adapted and
$\cF$-adapted. Then  \nir{the functors $L^i(\cG \circ \cF)$ and $\cG \circ L^i\cF$ are isomorphic.}
\end{lemma}

 \begin{definition}
 Let $G$ be an $l$-group. For any smooth representation $V$ of $G$ denote $\oH_i(G, V):=L^iCI_G(V)$.
Recall that $CI_G$ denotes the coinvariants functor.
 \end{definition}


\begin{proof}[Proof of Lemma \ref{LemShap}]
Note that $\cF(X) = ind_{G_x}^G \cF_x$. Note also that
$ind_{G_x}^G$ is an exact functor, and $CI_{G_x} = CI_{G} \circ
ind_{G_x}^G$. The lemma follows now from Lemma \ref{HomLeib}(i).
\end{proof}

\begin{lemma} \lbl{AFG}
Let $L$ be a \nir{lattice.} 
Let $V$ be a linear
space. Let $L$ act on $V$ by a character. Then
$$\oH_1(L,V) =  \oH_0(L,V) \otimes_{\bC}(L \otimes_{\bZ}\bC).$$
\end{lemma}
The proof of this lemma is straightforward.

\begin{lemma} \lbl{HomCoinv}
Let $L$ be an $l$-group and $L'<L$ be a subgroup. Then\\
(i) for any representation $V$ of $L$ we have
$$\oH_i(L',V)=L^i\mathcal{F}(V),$$ where $\mathcal{F}: \cM(L) \to
Vect$ is the functor defined by $\mathcal{F}(V)=V_{L'}$.\\
(ii) Suppose that $L'$ is normal. Let $\mathcal{F'}: \cM(L) \to
\cM(L/L')$ be the functor defined by $\mathcal{F'}(V)=V_{L'}$. Then
for any representation $V$ of $L$ we have
$\oH_i(L',V)=L^i\mathcal{F'}(V).$
\end{lemma}
\begin{proof}
(i) Consider the restriction functor $Res_{L'}^L: \cM(L) \to
\cM(L')$. Note that it is exact. Consider also the functor
$\cG:\cM(L') \to Vect$ defined by $\cG(\rho):=\rho_{L'}$. Note
that $\cF=\cG \circ Res_{L'}^L$. The
\Rami{assertion}
follows now from Lemma
\ref{HomLeib}(i)
\Rami{using the fact that $\Sc(L)$ is a projective object in $\cM(L')$}.\\
(ii) follows from (i) in a similar way, but using part (iii) of
Lemma \ref{HomLeib} instead part (i).
\end{proof}

\begin{lemma} \lbl{CuspAcyc}
Let $G$ be a reductive group and $H<G$ be a subgroup. Consider the
functor $$\mathcal{F}: \cM(G(F)) \to \cM(H(F)/(H(F)\cap G^1))
\text{ defined by }\mathcal{F}(V)=V_{H(F)\cap G^1}.$$ Then any \nir{finitely generated}
cuspidal representation of $G(F)$ is an $\mathcal{F}$-acyclic
object.
\end{lemma}
\begin{proof}
Consider the restriction functors $$Res_{1}^{H(F)/(H(F)\cap G^1)}:
\cM(H(F)/(H(F)\cap G^1)) \to Vect$$ and $$Res_{G^1}^{G(F)}:
\cM(G(F)) \to \cM(G^1).$$ Note that they are exact. Consider also
the functor $\cG:\cM(G^1) \to Vect$ defined by
$\cG(\rho):=\rho_{G^1 \cap H(F)}$. Denote $\cE:= \cG \circ
Res_{G^1}^{G(F)}$. Note that $\cE= Res_{1}^{H(F)/(H(F)\cap G^1)}
\circ \cF$.

$$\xymatrix{\parbox{40pt}{$\cM(G(F))$}\ar@{->}^{\cF \quad \quad
\quad \quad}[r]\ar@{->}^{\cE}[dr]\ar@{->}_{Res_{G^1}^{G(F)} }[d] &
\parbox{110pt}{$\cM(H(F)/(H(F)\cap
G^1))$}\ar@{->}^{Res_{1}^{H(F)\cap G^1}}[d]\\
\parbox{30pt}{$\cM(G^1)$}\ar@{->}^{\cG}[r] &
\parbox{25pt}{$Vect$}}$$

Let $\pi$ be a cuspidal finitely generated representation of $G(F)$. By Corollary
\ref{CuspProj}, $Res_{G^1}^{G(F)}(\pi)$ is projective and hence
$\cG$-acyclic. Hence by Lemma \ref{HomLeib}(ii) $\pi$ is
$\cE$-acyclic. Hence by Lemma \ref{HomLeib}(iii) $\pi$ is
$\cF$-acyclic.

\end{proof}

\begin{lemma} \lbl{HomQGroups}
Let $L$ be an l-group and $L'<L$ be a normal subgroup. Suppose
that $\oH_i(L',\bC)=0$ for all $i>0$. Let $\rho$ be a
representation of $L/L'$. Denote by $Ext(\rho)$ the natural
representation of $L$ obtained from $\rho$. Then
$\oH_i(L/L',\rho)=\oH_i(L,Ext(\rho))$.
\end{lemma}
\begin{proof}
Consider the coinvariants functors $\cE: \cM(L) \to Vect$ and
$\cF: \cM(L/L') \to Vect$ defined by $\cE(V):=V_L$ and
$\cF(V):=V_{L/L'}$. Note that $\cF = \cE \circ Ext$ and $Ext$ is
exact. By Shapiro Lemma (Lemma \ref{LemShap}), $\Sc(L/L')$ is acyclic with respect to
both $\cE$ and $\cF$. The lemma follows now from Lemma
\ref{HomLeib}(ii).
\end{proof}

\begin{remark}
Recall that if $L'=N(F)$ where $N$ is a unipotent algebraic group,
then $\oH_i(L')=0$ for all $i>0$.
\end{remark}

Now we are ready to prove Lemma \ref{FinDimH1H0}
\begin{proof}[Proof of Lemma \ref{FinDimH1H0}]
By Lemma \ref{HomQGroups} we can assume that $G$ is reductive.

Let $\cF:\cM(G(F)) \to Vect$ be the functor defined by
$\cF(V):=V_{H(F)}.$ Let $$\cG:\cM(G(F)) \to \cM(H(F)/(H(F)\cap
G^1))$$ be defined by $$\cG(V):=V_{H(F)\cap G^1}.$$ Let
$$\cE:\cM(H(F)/(H(F)\cap G^1)) \to Vect$$ be defined by
$$\cE(V):=V_{H(F)/(H(F)\cap G^1)}.$$ Clearly, $\cF=\cE \circ \cG$.
By Lemma \ref{CuspAcyc}, $\rho$ is $\cG$-acyclic. Hence by Lemma
\ref{HomLeib}(ii), $L^i\cF(\rho)=L^i\cE(\cG(\rho))$.

$$\xymatrix{\parbox{40pt}{$\cM(G(F))$}\ar@{->}^{\cG \quad \quad \quad \quad}[r] \ar@/^3pc/@{->}^{\cF}[rrr] &
\parbox{105pt}{$\cM(H(F)/(H(F)\cap G^1))$}\ar@/^2pc/@{->}^{\cE}[rr]\ar@{->}^{\cK}[r] &
\parbox{105pt}{$\cM(H(F)/(H(F)\cap G^0))$} \ar@{->}^{\quad \quad \quad \quad \cC}[r] &
\parbox{25pt}{$Vect$}}$$

Consider the coinvariants functors $\cK: \cM(H(F)/(H(F)\cap G^1))
\to \cM(H(F)/(H(F)\cap G^0))$ and $\cC:\cM(H(F)/(H(F)\cap G^0))
\to Vect$ defined by $\cK(\rho):=\rho_{(H(F) \cap G_0)/(H(F)\cap
G^1)}$ and $\cC(\rho):=\rho_{H(F)/(H(F)\cap G^1)}$. Note that $\cE
= \cC \circ \cK$.

 Note that
$\cC$ is exact since the group $H(F)/(H(F)\cap G^1)$ is finite.
Hence by Lemma \ref{HomLeib}(iii), $L^i\cE = \cC \circ L^i\cK$.

Now, by Lemma \ref{HomCoinv},
$$\oH_i(H(F),\rho)=L^i\cF(\rho)=L^i\cE(\cG(\rho))=
\cC(L^i\cK(\cG(\rho)))=\cC(\oH_i((H(F) \cap G_0)/(H(F)\cap
G^1),\cG(\rho))).$$

Hence, by Lemma \ref{AFG}, if $\oH_0(H(F),\rho)$ is finite
dimensional then $\oH_1(H(F),\rho)$ is finite dimensional.
\end{proof}

\section{Uniform Spherical Pairs} \lbl{sec:UniSPairs}

In this section we introduce the notion of uniform spherical pair
and prove Theorem \ref{thm:ModIso}.

\Dima{We follow the main steps of \cite{Kaz}, where the author constructs an isomorphism between the Hecke algebras of a reductive group over close enough local fields. First, he constructs a linear isomorphism between the Hecke algebras, using Cartan decomposition. Then, he shows that for two given elements of the Hecke algebra there exists $m$ such that if the fields are  $m$-close then the product of those elements will be mapped to the product of their images. Then he uses the fact that the Hecke algebras are finitely generated and finitely presented to deduce the theorem.}

\Dima{Roughly speaking, we call a pair $H<G$ of reductive groups a uniform spherical pair if it possesses a relative analog of Cartan decomposition, i.e. a ``nice'' description of the set of double cosets $K_0(G,F)\setminus G(F) / H(F)$ which in some sense does not depend on $F$. We give the precise definition in the first subsection and prove Theorem \ref{thm:ModIso} in the second subsection.}

\subsection{Definitions} \lbl{subsec:UniSpairs}$ $

Let $R$ be a complete and smooth local ring, let $m$ denote its maximal ideal, and let $\pi$ be an element in $m\setminus m^2$. A good example to keep in mind is the ring $\bZ_p[[\pi]]$.  An $(R,\pi)$-local field is a local field $F$ together with an epimorphism of rings $R\to O_F$, such that the image of $\pi$ (which we will continue to denote by $\pi$) is a uniformizer. Denote the collection of all $(R,\pi)$-local fields by $\cF_{R,\pi}$.

Suppose that $G$ is a  reductive group defined and split over $R$. Let $T$ be a fixed split torus, and let $X_*(T)$ be the coweight lattice of $T$. For every $\la\in X_*(T)$ and every $(R,\pi)$-local field $F$, we write $\pi^\la=\la(\pi)\in T(F)\subset G(F)$. We denote the subgroup $G(O_F)$ by $K_0(F)$, and denote its $\ell$'th congruence subgroup by $K_\ell(F)$.

\begin{defn}
Let $F$ be a local field.
Let $X \subset \bA^n_{O_F}$ be a closed subscheme. For any $x,y\in X(F)$, define the valuative distance between $x$ and $y$ to be $\val_F(x,y):=\min\{\val_F(x_i-y_i)\}$. Also, for
any  $x\in X(F)$, define $\val_F(x):=\min\{\val_F(x_i)\}$.
The ball of valuative radius $\ell$ around a point $x$ in
$X(F)$
 will be denoted by $B(x,\ell)(F)$.
\end{defn}

\begin{defn}\lbl{defn:good.pair}

Let $G$ be a split reductive group defined over $R$ and let $H\subset G$ be a smooth reductive subgroup defined over $R$. We say that the pair $(G,H)$ is uniform spherical if there are
\begin{itemize}
\item An $R$-split torus $T\subset G$,
\item An affine embedding $G/H\hookrightarrow \bA^n$.
\item A finite subset $\fX\subset G(R)/H(R)$.
\item A subset $\Upsilon\subset X_*(T)$.
\end{itemize} such that
\begin{enumerate}
\item The map $x\mapsto K_0(F)x$ from $\pi^\Upsilon\fX$ to $K_0(F)\bs G(F)/H(F)$ is onto for every $F\in\cF_{R,\pi}$. \lbl{cond:1}
\item For every $x,y \in \pi^\Upsilon \fX\subset (G/H)(R[\pi^{-1}])$, the closure in $G$ of the $R[\pi^{-1}]$-scheme  \Dima{$$T_{x,y}:=\{g\in G \times_{\Spec(R)} \Spec(R[\pi^{-1}]) | gx=y\}$$} is smooth over $R$. We denote this closure by $S_{x,y}$.
\lbl{cond:connectors}
\item  For every $x \in \pi^\Upsilon \fX$, the valuation $\val_F(x)$ does not depend on $F \in \cF_{R,\pi}.$
\RGRCor{\item  There exists $l_0$ s.t. for any $l>l_0$, for any
$F\in\cF_{R,\pi}$ and for every $x \in \fX$ and $\alpha \in
\Upsilon$ we have $K_l \pi^\alpha K_l x= K_l \pi^\alpha  x $
\lbl{cond:GR}.}
\end{enumerate}

\nir{If $G,H$ are defined over $\bZ$, we say that the pair $(G,H)$ is uniform spherical if, for every $R$ as above, the pair $(G\times_{\Spec(\bZ)}\Spec(R),H\times_{\Spec(\bZ)}\Spec(R))$ is uniform spherical.
}
\end{defn}

In Section \ref{sec:ap} we give two examples of uniform spherical pairs. We will list now several basic properties of uniform spherical pairs.
\RGRCor{In light of the recent developments in the structure theory of symmetric and spherical pairs (e.g. \cite{Del}, \cite{SV}),
we believe that the majority of symmetric pairs and many spherical pairs defined over local fields are specializations of appropriate uniform spherical pairs.}

From now and until the end of the section we fix a uniform spherical pair $(G,H)$. First note that, since $H$ is smooth, the fibers of $G\to G/H$ are smooth. Hence the map $G\to G/H$ is smooth.

\begin{lem}\lbl{lem:SO}  Let $(G,H)$ be a uniform spherical pair.
Let $x,y\in \pi^{\Upsilon} \fX$. Let $F$ be an $(R,\pi)$-local field. Then
$$S_{x,y}(O_F) = T_{x,y}(F) \cap G(O).$$
\end{lem}
\begin{proof}
The inclusion $S_{x,y}(O_F) \subset T_{x,y}(F) \cap G(O_F)$ is evident. In order to prove the other inclusion we have to show that any map $\psi: \Dima{\Spec(O_F)\to G \times _{\Spec R} \Spec O_F}$ such that $\Img (\psi|_{\Spec F}) \subset \Dima{T_{x,y} \times _{\Spec R[\pi^{-1}]} \Spec F}$ satisfies $\Img \psi \subset \Dima{S_{x,y} \times _{\Spec R} \Spec O_F}$.

This holds since $\Dima{S_{x,y} \times _{\Spec R} \Spec O_F}$ lies in the closure of $\Dima{T_{x,y} \times_{\Spec R[\pi^{-1}]} \Spec F}$ in $\Dima{G \times _{\Spec R} \Spec O_F}$.
\end{proof}

\begin{lem} \lbl{lem:rough.bijection} If $(G,H)$ is uniform spherical, then there is a subset $\De\subset \pi^\Upsilon \fX$ such that, for every $F\in\cF_{R,\pi}$, the map $x\mapsto K_0(F)x$ is a bijection between $\De$ and $K_0(F)\bs G(F)/H(F)$.
\end{lem}

\begin{proof} It is enough to show that for any $F,F'\in\cF_{R,\pi}$ and for any $x,y\in \pi^\Upsilon \fX$, the equality $K_0(F)x=K_0(F)y$ is equivalent to $K_0(F')x=K_0(F')y$.

The scheme $S_{x,y}\otimes O_F$ is smooth over $R$, and hence it is smooth over $O_F$. Therefore, it is formally smooth. This implies that the map $S_{x,y}(O_F) \to S_{x,y}(\bF_q)$ is onto and hence $\{g\in G(O_F)| gx=y\}\neq\emptyset$ if and only if $S_{x,y}(\bF_q)\neq\emptyset$.

Hence, the two equalities
$K_0(F)x=K_0(F)y$ and
$K_0(F')x=K_0(F')y$ are equivalent
to $S_{x,y}(\bF_q)\neq\emptyset$.
\end{proof}

From now untill the end of the section we fix $\De$ as in the lemma.

\begin{prop} \lbl{prop:cond.ball} If $(G,H)$ is uniform spherical, then for every $x \in\pi^\Upsilon \fX$ and every $\ell\in\bN$, there is $M\in\bN$ such that for every $F\in\cF_{R,\pi}$, the set $K_\ell(F)x$ contains a ball of radius $M$ around $x$.
\end{prop}

\begin{proof}
Since, for every $\de\in X_*(T)$ and every $\ell\in\bN$, there is $n\in\bN$ such that $K_n(F) \subset \pi^\de K_\ell(F)\pi^{-\de}$ for every $F$, we can assume that $x \in \fX$. The claim now follows from the following version of the implicit function theorem.
\begin{lem} \lbl{lem:ball.crit}
Let $F$ be a local field. Let $X$ and $Y$ be affine schemes defined over $O_F$. Let $\psi:X \to Y$ be a smooth morphism defined
over $O_F$. Let $x \in X(O_F)$ and $y:=\psi(x)$. Then
$\psi(B(x,\ell)(F)) = B(y,\ell)(F)$ for any natural number $l$.
\end{lem}

\begin{proof}
The inclusion $\psi(B(x,\ell)(F)) \subset B(y,\ell)(F)$ is clear. We prove the inclusion
$\psi(B(x,\ell)(F)) \supset B(y,\ell)(F).$\\
Case 1: $X$ and $Y$ are affine spaces and $\psi$ is etale. The proof is standard.\\
Case 2: $X=\bA^m$, $\psi$ is etale: We can assume that $Y\subset \bA^{m+n}$ is defined by $f_1,\ldots,f_n$ with independent differentials, and that $\psi$ is the projection. The proof in this case follows from Case 1 by considering the map $F:\bA^{m+n}\to\bA^{m+n}$ given by $F(x_1,\ldots,x_{m+n})=(x_1,\ldots,x_m,f_1,\ldots,f_n)$.\\
Case 3: $\psi$ is etale: Follows from Case 2 by restriction from the ambient affine spaces.\\
Case 4: In general, a smooth morphism is a composition of an etale morphism and a projection, for which the claim is trivial.
\end{proof}
\end{proof}

\begin{lem} \lbl{lem:cond.fin} For every $\la \in X_*(T)$ and $x \in \pi^{\Upsilon}\fX$, there is a finite subset $B \subset\pi^{\Upsilon}\fX$ such that $\pi^\la K_0(F)x\subset\bigcup_{y\in B} K_0(F)y$ for all $F\in\cF_{R,\pi}$.
\end{lem}

\begin{proof}
By Lemma \ref{lem:rough.bijection}, we can assume that the sets $K_0(F)\pi^\la x_0$ for $\la\in\Upsilon$ are disjoint. There is a constant $C$ such that for every $F$ and for every $g\in\pi^\la K_0(F)\pi^\de$, $val_F(gx_0)\geq C$. Fix $F$ and assume that $g\in K_0(F)\pi^\la K_0(F)\pi^\de$. From the proof of Proposition \ref{prop:cond.ball}, it follows that $K_0(F)gx_0$ contains a ball whose radius depends only on $\la,\de$. Since $F$ is locally compact, there are only finitely many disjoint such balls in the set $\{x\in G(F)/H(F) \,|\, val_F(x)\geq C\}$, so there are only finitely many $\eta\in\Upsilon$ such that $val_F(\pi^\la x_0) \geq C$. By definition, this finite set, $S$, does not depend on the field $F$. Therefore, $\pi^\la K_0(F)\pi^\de x_0\subset\bigcup_{\eta\in S} K_0(F)\pi^\eta x_0$.
\end{proof}
\RGRCor{
\begin{notation}
$ $
\begin{itemize}
\item Denote by $\cM_\ell(G(F)/H(F))$ the space of $K_\ell(F)$-invariant
compactly supported measures on $G(F)/H(F)$. 
\item  For a $K_l$ invariant subset $U \subset G(F)/H(F)$ we denote by $1_U \in \cM_\ell(G(F)/H(F))$  the Haar measure on $G(F)/H(F)$
 multiplied by the characteristic function of $U$ and normalized s.t. its integral is $1.$ We define in a similar way $1_V \in \cH_{\ell}(G,F)$ for a $K_l$-double  invariant subset $V \subset G(F)$.
\end{itemize}

\end{notation}
}

\RGRCor{
\begin{prop} \lbl{prop:FG}
If $(G,H)$ is uniform spherical then $\cM_\ell(G(F)/H(F))$ is finitely generated over $\cH_{\ell}(G,F)$ for any $\ell$.
 \end{prop}
\begin{proof}
As in step 4 of Lemma \ref{VKFinGen},  it is enough to prove the assertion for large enough $l$. Thus we may assume that for every $x \in \fX$ and $\alpha \in
\Upsilon$ we have $K_l \pi^\alpha K_l x= K_l \pi^\alpha  x $. Therefore, 
$1_{K_l \pi^\alpha K_l}1_{K_l x}=1_{K_l \pi^\alpha  x}$. Hence for any $g\in K_0/K_l$ we have 
$(g 1_{K_l \pi^\alpha K_l})1_{K_l x}=1_{gK_l \pi^\alpha  x}$. \DimaGRCor{Now, the elements $1_{gK_l \pi^\alpha  x}$ span $\cM_\ell(G(F)/H(F))$ by condition \ref{cond:1} in definition \ref{defn:good.pair}. } This implies that the elements $1_{K_l x}$ generate $\cM_\ell(G(F)/H(F))$.
\end{proof}
}

\subsection{Close Local Fields}\lbl{subsec:CloseLocalFields}

\begin{defn} Two $(R,\pi)$-local fields $F,E\in\cF_{R,\pi}$, are $n$-close if there is an isomorphism $\phi_{E,F}:O_F/\pi^n\to O_E/\pi^n$ such that the two maps $R\to O_F\to O_F/\pi^n\to O_E/\pi^n$ and $R\to O_E\to O_E/\pi^n$ coincide. In this case, $\phi$ is unique.
\end{defn}

\begin{thm}[\cite{Kaz}]
Let $F$ be an $(R,\pi)$ local field. Then, for any $\ell$, there exists $n$ such that, for any $E\in\cF_{R,\pi}$, which is $n$-close to $F$, there exists a unique isomorphism $\Phi_{\cH,\ell}$ between the algebras $\cH_{\ell}(G,F)$ and $\cH_{\ell}(G,E)$ that maps the Haar measure on $K_{\ell}(F)\pi^\la K_{\ell}(F)$ to the  Haar measure on $K_{\ell}(E)\pi^\la K_{\ell}(E)$, for every $\la\in X_*(T)$, and intertwines the actions of the finite group $K_0(F)/K_{\ell}(F)\stackrel{\phi_{F,E}}{\cong}K_0(E)/K_{\ell}(E)$.
\end{thm}

In this section we prove the following refinement of Theorem \ref{thm:ModIso} from the Introduction:

\begin{thm} \lbl{thm:phiModIso} Let $(G,H)$ be a uniform spherical pair.
\Removed{
Suppose that the module $\cM_\ell(G(F)/H(F))$ is finitely generated over the Hecke algebra $\cH(G(F),K_\ell(F))$ for any $F\in\cF_{R,\pi}$ and $\ell \in \bN$. }
Then, for any $\ell \in \bN$ and $F\in\cF_{R,\pi}$, there exists $n$ such that, for any $E\in\cF_{R,\pi}$ that is $n$-close to $F$, there exists a unique map $$\cM_\ell(G(F)/H(F)) \to \cM_\ell(G(E)/H(E))$$ which is an isomorphism of modules over the Hecke algebra $$\cH(G(F),K_\ell(F))\stackrel{\Phi_{\cH,\ell}}{\cong} \cH(G(E),K_\ell(E))$$ that maps the Haar measure on $K_\ell(F)x$ to the Haar measure on $K_\ell(E)x$, for every $x\in\De\subset \pi^\la\Upsilon$, and intertwines the actions of the finite group $K_0(F)/K_\ell(F)\stackrel{\phi_{F,E}}{\cong}K_0(E)/K_\ell(E)$.
\end{thm}

For the proof we will need notation and several lemmas.

\begin{notation}
For any valued field $F$ with uniformizer $\pi$ and any integer $m \in \bZ$, we
denote by $\res_m:F \to F / \pi^m O$ the projection. Note that the
groups $\pi^n O$ are naturally isomorphic for all $n$.
Hence if two local fields $F,E\in\cF_{R,\pi}$ are $n$-close,
then for any $m$ we are given an isomorphism, which we also denote by $\phi_{F,E}$ between $\pi^{m-n} O_F/\pi^{m}O_F$ and $\pi^{m-n} O_{E} /\pi^{m}O_{E}$, which are subgroups of
$F/\pi^{m}O_F$ and $E/\pi^{m}O_{E}$.
\end{notation}

\begin{lem} \lbl{lem:equal.stabilizers} Suppose that $(G,H)$ is a uniform spherical pair, and suppose that $F,E\in\cF_{R,\pi}$ are $\ell$-close. Then for all $\de\in\De$,
\[
\phi_{F,E}(\Stab_{K_0(F)/K_\ell(F)}K_\ell(F) \de )=\Stab_{K_0(E)/K_\ell(E)}K_\ell(E) \de.
\]
\end{lem}

\begin{proof} The stabilizer of $K_\ell(F)\de$ in
$K_0/K_\ell$ is the projection of the stabilizer of $\de$
in $K_0$ to $K_0/K_\ell$. In other words, it is the image of
$S_{\de,\de}(O_F)$ in $S_{\de,\de}(O_F/\pi^\ell)$.
Since
$S_{\de,\de}$ is smooth over $R$, it is smooth over $O_F$. Hence $S_{\de,\de}$ is formally smooth, and so
this map is onto. The same applies to the stabilizer of
$K_\ell(E)\de$ in $K_0(E)/K_\ell(E)$, but
$\phi_{F,E}(S_{\de,\de}(O/\pi^\ell))=S_{\de,\de}(O'/\pi'^\ell)$.
\end{proof}

\begin{cor}\lbl{cor:LinIso}
Let $\ell \in \bN$.
\Removed{Suppose that the module $\cM_\ell(G(F)/H(F))$ is finitely generated over the Hecke algebra $\cH(G(F),K_\ell(F))$ for any $F\in\cF_{R,\pi}$. }

Then, for any $F,E\in\cF_{R,\pi}$ that are $\ell$-close, there exists a unique \Dima{morphism }
of vector spaces $$\Phi_{\cM,\ell}:\cM_\ell(G(F)/H(F)) \to \cM_\ell(G(E)/H(E))$$ that maps the Haar measure on $K_\ell(F)x$ to the Haar measure on $K_\ell(E)x$, for every $x\in\De$, and intertwines the actions of the finite group $K_0(F)/K_{\ell}(F)\stackrel{\phi_{F,E}}{\cong}K_0(E)/K_{\ell}(E)$. \Dima{Moreover, this morphism is an isomorphism.}
\end{cor}
\begin{proof}
The uniqueness is evident.
By Lemma \ref{lem:equal.stabilizers} and Lemma \ref{lem:rough.bijection},
the map between $K_\ell(F)\bs G(F)/H(F)$ and $K_\ell(E)\bs G(E)/H(E)$ given by
\[
K_\ell(F) g\de \mapsto K_\ell(E) g'\de,
\]
where $g\in K_0(F)$ and $g'\in K_0(E)$ satisfy that
$\phi_{F,E}(\res_\ell(g))=\res_\ell(g')$, is a bijection.
This  bijection gives the required isomorphism.
\end{proof}

\begin{remark}
A similar construction can be applied to the pair $(G\times G,\Delta G)$. In this case, the main result of \cite{Kaz} is that the obtained linear map $\Phi_{\cH,\ell}$ between the Hecke algebras $\cH(G(F),K_\ell(F))$ and $\cH(G(E),K_\ell(E))$ is an isomorphism of algebras if the fields $F$ and $E$ are close enough.
\end{remark}

The following Lemma is evident:
\begin{lem} Let $P(x)\in R[\pi^{-1}][x_1,\ldots,x_d]$ be a polynomial. For any natural numbers $M$ and $k$, there is $N$ such that, if $F,E\in\cF_{R,\pi}$ are $N$-close, and $x_0\in \pi^{-k}O_F^d,y_0\in \pi^{-k}O_E^d$ satisfy that $P(x_0)\in \pi^{-k}O_F$ and $\phi_{F,E}(\res_{N}(x_0))=\res_{N}(y_0)$, then $P(y_0)\in \pi^{-k}O_E$ and $\phi_{F,E}(\res_{M}(P(x_0)))=\res_{M}(P(y_0))$.
\end{lem}

\begin{cor} Suppose that $(G,H)$ is a uniform spherical pair.
Fix an embedding of $G/H$ to an affine space $\bA^d$. Let $\la \in X_*(T)$, $x \in \pi^{\Upsilon}\fX$, $F\in\cF_{R,\pi}$, and  $k \in G(O_F)$. Choose $m$ such that $\pi^{\la} k x \in \pi^{-m}O_F^d$. Then, for every $M$, there is $N\geq M+m$ such that, for any $E\in\cF_{R,\pi}$ that is $N$-close to $F$, and for any $k' \in G(O_E)$ such that $\phi_{F,E}(res_N(k))=res_N(k')$,
\[
\pi^{\la} k' x \in G(E)/H(E) \cap \pi^{-m}O_E^d \text{ and } \phi_{F,E}(res_M(\pi^{\la} k x))=res_M(\pi^{\la} k' x).
\]
\end{cor}

\begin{cor} Suppose that $(G,H)$ is a uniform spherical pair.
Fix an embedding of $G/H$ to an affine space $\bA^d$. Let $m$ be an integer. For every $M$, there is $N$ such that, for any $F,E\in\cF_{R,\pi}$ that are $N$-close, any $x \in G(F)/H(F) \cap \pi^{-m}O_F^d$ and any $y\in G(E)/H(E) \cap \pi^{-m}O_E^d$, such that $\phi_{F,E}(res_{N-m}(x))=res_{N-m}(y)$, we have $\Phi_{\cM}(1_{K_M(F)} x) = 1_{K_M(E)} y$.
\end{cor}

\begin{proof}
Let $k_F \in G(O_F)$ and $\delta \in \Delta$ such that $x=k_F\de$. By Proposition \ref{prop:cond.ball}, there is an $l$ such that, for any $L\in\cF_{R,\pi}$ and any $k_L \in
G(O_L)$, we have $K_M(L)k_L\de$ contains a ball of radius $l$.

Using the previous corollary, choose an integer $N$ such that, for any $F$ and $E$ that are $N$-close and any $k_E \in G(O_E)$, such that $\phi_{F,E}(res_N(k_F))=res_N(k_E)$, we have
\[
k_E \delta \in (G(E)/H(E)) \cap \pi^{-m}O_E^d \text{ and }\phi_{F,E}(res_l(x))=res_l(k_E \delta).
\]
Choose such $k_E \in G(O_E)$ and let $z = k_E \delta$. Since $\res_l(z)=\phi_{F,E}(res_l(x))=res_l(y)$, we have that $z \in B(y,l)$, and hence $z \in K_M(E)y$. Hence

\[
1_{K_M(E)}y=1_{K_M(E)}z = \Phi_{\cM}(1_{K_M(F)} x).
\]
\end{proof}

From the last two corollaries we obtain the following one.
\begin{cor}\lbl{cor:CharFun} Given $\ell \in \bN$, $\la \in X_*(T)$, and $\de \in \Delta$, there is $n$ such that if $F,E\in\cF_{R,\pi}$ are $n$-close, and $g_F\in G(O_F),g_E\in G(O_E)$ satisfy that $\phi_{F,E}(\res_n(g_F))=\res_n(g_E)$, then $\Phi_{\cM,\ell}(1_{K_\ell(F)} \pi^\la g_F \de)=1_{K_\ell(E)}\pi^\la g_E \de$.
\end{cor}

\begin{prop} \lbl{prop:local.isom} Let $F\in\cF_{R,\pi}$. Then for every $\ell$, and every two elements $f\in\cH_\ell(F)$ and $v\in\cM_\ell(F)$, there is $n$ such that, if $E\in\cF_{R,\pi}$ is $n$-close to $F$, then $\Phi_{\cM,\ell}(f\cdot v)=\Phi_{\cH,\ell}(f)\cdot\Phi_{\cM,\ell}(v)$.
\end{prop}

\begin{proof} By linearity, we can assume that $f = 1_{K_\ell(F)} k_1 \pi^\la k_2
1_{K_\ell(F)}$ and that $v= 1_{K_\ell(F)} k_3\de$, where
$k_1,k_2,k_3\in K_0(F)$. Choose $N \geq l$ big enough so that
$\pi^\la K_N(F) \pi^{-\la} \subset K_\ell(F)$.

Choose $k_i'\in G(O_E)$ such that $\phi_{F,E}(\res_N(k_i))=\res_N(k_i')$.
Since $\Phi_{\cM,\ell}$ and $\Phi_{\cH,\ell}$ intertwine left multiplication by
$1_{K_\ell(F)} k_1 1_{K_\ell(F)}$ to left multiplication by $1_{K_\ell(E)}
k'_1 1_{K_\ell(E)}$, we can assume that $k_1=1=k'_1$. Also, since
$k_2$ normalizes $K_\ell(F)$, we can assume that $k_2=1=k'_2$. Let
$K_\ell(F)=\bigcup_{i=1}^s K_N(F) g_i$ be a decomposition of $K_\ell(F)$
into cosets. Choose $g'_i\in K_\ell(E)$ such that
$\phi_{F,E}(\res_N(g_i))=\res_N(g'_i)$. Then
\[
1_{K_\ell(F)}=c\sum_{i=1}^s 1_{K_N(F)} g_i \text{ and } 1_{K_\ell(E)}=c\sum_{i=1}^s 1_{K_N(E)} g'_i
\]
where $c=|K_\ell(F)/K_N(F)|=|K_\ell(E)/K_N(E)|$. Hence
\[
fv=1_{K_\ell(F)} \pi^\la 1_{K_\ell(F)} k_3 \de =c\sum_{i=1}^s
1_{K_\ell(F)} \pi^\la 1_{K_N(F)} g_i k_3 \de =c\sum_{i=1}^s
1_{K_\ell(F)} \pi^\la g_i k_3 \de .
\]
and
\[
\Phi_{\cH,\ell}(f)\Phi_{\cM,\ell}(v)=1_{K_\ell(E)} \pi^\la 1_{K_\ell(E)} k'_3 \de=c\sum_{i=1}^s 1_{K_\ell(E)} \pi^\la 1_{K_N(E)} g_i k'_3\de =c\sum_{i=1}^s 1_{K_\ell(E)} \pi^\la g'_i k'_3 \de.
\]
The proposition follows now from Corollary \ref{cor:CharFun}.
\end{proof}

Now we are ready to prove Theorem \ref{thm:phiModIso}.
\begin{proof}[Proof of Theorem \ref{thm:phiModIso}]
We have to show for any $\ell$ there exists $n$ such that if $F,E\in \cF_{R,\pi}$ are $n$-close then the map
$\Phi_{\cM,l}$ constructed in Corollary \ref{cor:LinIso} is an isomorphism of modules over $\cH(G(F),K_\ell(F))\stackrel{\Phi_{\cH,\ell}}{\cong} \cH(G(E),K_\ell(E))$.

Since $\cH(G(F),K_\ell(F))$ is Noetherian, $\cM_\ell(G(F)/H(F))$ is generated by a finite set $v_1,\ldots,v_n$ satisfying a finite set of relations $\sum_i f_{i,j}v_i=0$.
\RGRCor{Without loss of generality we may assume that for any $x \in \fX$ the Haar measure on $K_\ell(F)x$ is contained in the set $\{v_i\}.$}

By Proposition \ref{prop:local.isom}, if $E$ is close enough to $F$, then $\Phi_{\cM,\ell}(v_i)$ satisfy the above relations.

Therefore there exists a  homomorphism of Hecke modules $\Phi':\cM_\ell(G(F)/H(F))\to\cM_\ell(G(E)/H(E))$ given on the generators $v_i$ by $\Phi'(v_i):= \Phi_{\cM,\ell}(v_i)$. \\ \RGRCor{
$\Phi'$ intertwines the actions of the finite group $K_0(F)/K_{\ell}(F)\stackrel{\phi_{F,E}}{\cong}K_0(E)/K_{\ell}(E)$. Therefore, by Corollary \ref{cor:LinIso}, in order to show that $\Phi'$ coincides with  $\Phi_{\cM,\ell}$ it is enough to check that $\Phi'$ maps the normalized Haar measure on $K_\ell(F)x$ to the normalized Haar measure on $K_\ell(E)x$ for every $x\in\De$. In order to  do this let us decompose $x=\pi^\alpha x_0$ where
$x_{0} \in \fX$ and $\alpha \in \Upsilon.$  Now, since $(G,H)$ is uniformly spherical we have $$1_{K_n(F) x}=1_{K_n(F) \pi^\alpha K_n(F)}1_{K_n(F) x_0}$$ and $$1_{K_n(E) x}=1_{K_n(E) \pi^\alpha K_n(E)}1_{K_n(E) x_0}.$$ Therefore, since $\Phi'$ is a homomorphism, we have $$\Phi'(1_{K_n(F) x})=\Phi'(1_{K_n(F) \pi^\alpha K_n(F)}1_{K_n(F) x_0})= 1_{K_n(E) \pi^\alpha K_n(E)}1_{K_n(E) x_0}=1_{K_n(F) x}.$$
}

Hence the linear map $\Phi_{\cM,\ell}:\cM_\ell(G(F)/H(F))\to\cM_\ell(G(E)/H(E))$ is a homomorphism of Hecke modules. Since it is a linear isomorphism, it is an isomorphism of Hecke modules.
\end{proof}

%
Now we obtain \Dima{the following generalization of }Corollary \ref{cor:GelGel}:

\begin{cor}\lbl{cor:GelGel2}
Let $(G,H)$ be a uniform spherical pair. Suppose that

\begin{itemize}
\item For any $F\in\cF_{R,\pi}$, the pair $(G,H)$ is $F$-spherical.
\Removed{
\item For any $F\in\cF_{R,\pi}$ and any irreducible smooth representation $\rho$ of $G(F)$ we have
$$\dim\Hom_{H(F)}(\rho|_{H(F)}, \bC) < \infty .$$}
\item
\Dima{For any $E\in\cF_{R,\pi}$ and natural number $n$, there is a field $F\in\cF_{R,\pi}$ such that $E$ and $F$ are $n$-close and the pair $(G(F),H(F))$ is a Gelfand pair, i.e. for any irreducible smooth representation $\rho$ of $G(F)$ we have
$$\dim\Hom_{H(F)}(\rho|_{H(F)}, \bC) \leq 1 .$$}
\end{itemize}
Then $(G(F),H(F))$ is a Gelfand pair for any $F\in\cF_{R,\pi}$.
\end{cor}

\nir{\begin{rem} Fix a prime power $q=p^k$. Let $F$ be the unramified extension of $\bQ_p$ of degree $k$, let $W$ be the ring of integers of $F$, and let $R=W[[\pi]]$. Then \Dima{$\cF_{R,\pi}$ includes all local fields with residue field $\bF_q$,} and so Corollary \ref{cor:GelGel2} implies Corollary \ref{cor:GelGel}.
\end{rem}}

Corollary \ref{cor:GelGel2} follows from \Dima{Theorem \ref{thm:phiModIso}}, 
Theorem
\ref{FinGen}, and the following lemma.
\begin{lem}
Let $F$ be a local field and $H<G$ be a pair of reductive groups
defined  over $F$. \Dima{Suppose that $G$ is split over $F$.}
Then $(G(F),H(F))$ is a Gelfand pair if and only
if for any \Dima{large enough }$l \in \bZ_{>0}$ and any simple module $\rho$ over
$\cH_l(G(F))$ we have
$$ \dim \Hom_{\cH_l(G(F))} (\cM_l(G(F)/H(F)), \rho) \leq 1.$$
\end{lem}

This lemma follows from statement (\ref{1}) formulated in Subsection
\ref{subsec:prel}.

\section{Applications}
\lbl{sec:ap}
\setcounter{thm}{0}

In this section we prove that the pair
$(\GL_{n+k}(F),\GL_n(F) \times\GL_k(F))$ is a Gelfand pair for any
local field $F$ of characteristic different from 2 and the pair
$(\GL_{n+1}(F),\GL_n(F) )$ is a strong Gelfand pair for any local
field $F$. We use \Dima{Corollary \ref{cor:GelGel2} } 
to deduce those results from the
characteristic zero case which were proven in \cite{JR} and
\cite{AGRS} respectively.
\Dima{Let $R=W[[\pi]]$.}

To verify condition (\ref{cond:connectors}) in Definition \ref{defn:good.pair}, we use the following straightforward lemma:

\begin{lem} \lbl{lem:SmoothCrit}
Let $G= (\GL_{n_1})_R \times \cdots \times (\GL_{n_k})_R$ and let $C<G \otimes_R R[\pi^{-1}]$
be a sub-group scheme defined over $R[\pi^{-1}]$. Suppose that $C$ is defined by equations of the
following type:
$$ \sum_{i=1}^{l} \ep_i a_{\mu_i} \pi^{\lambda_i} = \pi^{\nu},$$
or
$$ \sum_{i=1}^{l} \ep_i a_{\mu_i} \pi^{\lambda_i} = 0,$$
where $\ep_i = \pm 1$, $a_1,...,a_{n_1^2+...+n_k^2}$ are entries
of matrices, $1 \leq \mu_i \leq n_1^2+...+n_k^2$ are some indices,
and $\nu,\lambda_i$ are integers. Suppose also that the indices
$\mu_i$ are distinct for all the equations. Then the closure
$\overline{C}$ of $C$ in $G$ is smooth over $R$.
\end{lem}

\DimaGRCor{To verify condition (\ref{cond:GR}) in Definition
\ref{defn:good.pair}, we use the following straightforward lemma:
\begin{lem} \lbl{lem:GRCrit}
Suppose that there exists a natural number $\ell_0$ such that, for any $F\in\cF_{R,\pi}$ and any $\ell>\ell_0$, there is a subgroup  $P_\ell<K_\ell(G,F)$ satisfying that 
for every $x \in \fX$ 
\begin{enumerate}
%
 \item For any $\alpha \in \Upsilon$ we have $\pi^\alpha P_\ell
 \pi^{-\alpha} \subset K_\ell$.
\item $K_\ell x = P_\ell x$.
\end{enumerate}
Then condition (\ref{cond:GR}) in Definition \ref{defn:good.pair}
is satisfied.
\end{lem}}


In our applications, we use the following to show that the pairs we consider are $F$-spherical.
\begin{prop} \label{peop:F.sph} Let $F$ be an infinite field, and consider $G=\GL_{n_1}\times\cdots\times\GL_{n_k}$ embedded in the standard way in $M=\M_{n_1}\times\cdots\times\M_{n_k}$. Let $A,B\subset G \otimes F$ be two $F$-subgroups whose closures in $M$ are affine subspaces $M_A,M_B$.
\begin{enumerate}
\item For any $x,y\in G(F)$, if the variety $\{(a,b)\in A\times B | axb=y\}$ is non-empty, then it has an $F$-rational point.
\item If $(G,A)$ is a spherical pair, then it is also an $F$-spherical pair.
\end{enumerate}

\end{prop}
\begin{proof} \begin{enumerate}
\item Denote the projections $G\to\GL_{n_j}$ by $\pi_j$. Assume that $x,y\in G(F)$, and there is a pair $(\overline{a},\overline{b})\in (A\times B)(\overline{F})$ such that $\overline{a}x\overline{b}=y$. Let $L\subset M_A\times M_B$ be the \Dima{affine} subspace $\{(\al,\be)| \al x=y\be\}$, defined over $F$. By assumption, the functions $(\al,\be)\mapsto \det\pi_j(\al)$ and $(\al,\be)\mapsto\det\pi_j(\be)$, for $j=1,\ldots,k$, are non-zero on $L(\overline{F})$. Hence there is $(a,b)\in L(F)\cap G$, which means that $axb^{-1}=y$.

\item Applying (1) to $A$ and \Dima{any} parabolic subgroup $B\subset G$, any $(A\times B)(\overline{F})$-orbit in $G(\overline{F})$ contains at most one $(A\times B)(F)$-orbit. Since there are only finitely many $(A\times B)(\overline{F})$-orbits in $G(\overline{F})$, the pair $(G,A)$ is $F$-spherical.
\end{enumerate}
\end{proof}

\subsection{The Pair $(\GL_{n+k},\GL_n \times\GL_k)$}
\lbl{subsec:JR}$ $

In this subsection we assume $p\neq 2$ and consider only local fields of characteristic
different from 2.

Let $G:=(\GL_{n+k})_R$ and $H:=(\GL_n)_R \times (\GL_k)_R < G$ be the subgroup of
block matrices. Note that $H$ is a symmetric subgroup since it
consists of fixed points of conjugation by $\ep =
\begin{pmatrix}
  Id_{k} & 0 \\
  0 & -Id_{n}
\end{pmatrix}$. We prove that $(G,H)$ is a Gelfand pair
using Corollary \ref{cor:GelGel}.
The pair $(G,H)$ is a symmetric pair, hence it is a spherical pair and therefore by Proposition \ref{peop:F.sph} it is $F$-spherical.
The second condition of Corollary \ref{cor:GelGel}
is \cite[Theorem 1.1]{JR}. It
remains to prove that $(G,H)$ is a uniform spherical pair.

\begin{prop} \lbl{prop:JR}
The pair $(G,H)$ is uniform spherical.
\end{prop}

\begin{proof}

Without loss of generality suppose that $n \geq k$.
Let $\fX=\{x_0\}$, where

$$
x_0:=\begin{pmatrix}
  Id_{k} & 0 & Id_{k} \\
  0 & Id_{n-k} & 0\\
  0 & 0        & Id_k
\end{pmatrix} H \text{ and } \Upsilon=\{(\mu_1,...,\mu_k,0,...,0) \in  X_*(T) \, | \,  \mu_1 \leq ... \leq
\mu_k \leq 0\}.$$

To show the first condition we show that every double
coset in $K_0\bs G/H$ includes an element of the form

$\begin{pmatrix}
  Id_{k} & 0 & diag(\pi^{\mu_1},...,\pi^{\mu_k}) \\
  0 & Id_{n-k} & 0\\
  0 & 0        & Id_k
\end{pmatrix}\text{ s.t. } \mu_1 \leq ... \leq
\mu_k \leq 0.$ Take any $g \in G$. By left multiplication by $K_0$
we can bring it to upper triangular form. By right multiplication
by $H$ we can bring it to a form $\begin{pmatrix}
  Id_{n} & A \\
  0 & Id_{k}
\end{pmatrix}$. Conjugating by a matrix $\begin{pmatrix}
  k_1 & 0 \\
  0 & k_2
\end{pmatrix}\in K_0 \cap H$ we can replace it by $\begin{pmatrix}
  Id_{n} & k_1Ak_2^{-1} \\
  0 & Id_{k}
\end{pmatrix}$. Hence we can bring $A$ to be a $k$-by-$(n-k)$ block of zero, followed by the a diagonal matrix of the form $diag(\pi^{\mu_1},...,\pi^{\mu_k})$ s.t. $\mu_1 \leq ... \leq
\mu_k .$ Multiplying by an element of
$K_0$ of the form $\begin{pmatrix}
  Id_{k} & 0 & k \\
  0 & Id_{n-k} & 0\\
  0 & 0        & Id_k
\end{pmatrix}$ we can bring $A$ to the desired form. \\

As for the second condition, we first compute the stabilizer ${G}_{x_0}$ of $x_0$ in $G$.
Note that the coset $x_0 \in G/H$ equals
$$ \left \{
\begin{pmatrix}
  g_1 & g_2 & h \\
  g_3 & g_4 & 0\\
  0 & 0        & h
\end{pmatrix}
 | \begin{pmatrix}
  g_1 & g_2 \\
  g_3 & g_4
\end{pmatrix} \in (\GL_n)_R,\, h \in (\GL_k)_R \right \}$$
and
$$\begin{pmatrix}
  A & B & C\\
  D & E & F\\
  G & H & I
\end{pmatrix}
\begin{pmatrix}
  Id_{k} & 0 & Id_{k} \\
  0 & Id_{n-k} & 0\\
  0 & 0        & Id_k
\end{pmatrix} = \begin{pmatrix}
  A & B & A+C\\
  D & E & D+F\\
  G & H & G+I
\end{pmatrix}.$$
Hence $$G_{x_0}= \left \{ \begin{pmatrix}
  g_1 & g_2 & h - g_1 \\
  g_3 & g_4 & -g_3\\
  0 & 0        & h
\end{pmatrix}
 | \begin{pmatrix}
  g_1 & g_2 \\
  g_3 & g_4
\end{pmatrix} \in (\GL_n)_R,\, h \in (\GL_k)_R \right \}.$$
Therefore, for any $\delta_1 = (\lambda_{1,1},...,\lambda_{1,k},0,...,0), \delta_2 = (\lambda_{2,1},...,\lambda_{2,k},0,...,0)\in
\Upsilon$,

\begin{multline*}
{G(F)}_{\pi^{\la_1}x_0,\pi^{\la_2}x_0} = \left \{ \begin{pmatrix}
  \pi^{\la_2}g_1\pi^{-\la_1} & \pi^{\la_2}g_2 & \pi^{\la_2}(h - g_1) \\
  g_3\pi^{-\la_1} & g_4 & -g_3\\
  0 & 0        & h
\end{pmatrix}\, |\,
\begin{pmatrix}
  g_1 & g_2 \\
  g_3 & g_4
\end{pmatrix} \in (\GL_n)_R,\, h \in (\GL_k)_R \right \}=\\
= \left \{ \begin{pmatrix}
  A & B & C \\
  D & E & F\\
  0 & 0        & I
\end{pmatrix}\in GL_{n+k}, |D=-F\pi^{-\la_1}, C=\pi^{\la_2}I-A\pi^{\la_1} \right \}.
\end{multline*}

The second condition of Definition \ref{defn:good.pair} follows now from Lemma
\ref{lem:SmoothCrit}.\\

As for the third condition, we use the embedding $G/H\to G$ given by $g\mapsto g\ep g^{-1} \ep$. It is easy to see that $val_F(\pi^\mu x_0)=\mu_1,$ which
is independent of $F$.

\DimaGRCor{ Let us now prove the last condition using Lemma
\ref{lem:GRCrit}. Take $l_0=1$ and $$P:= \left \{ \begin{pmatrix}
  Id & 0 & 0 \\
  D & E & F\\
  G & H        & I
\end{pmatrix}\in GL_{n+k} \right \}.$$
Let $P_l:=P(F)\cap K_l(GL_{n+k},F).$
The first condition of Lemma \ref{lem:GRCrit} obviously holds. To
show the second condition, we have to show that for any $F$, any
$l \geq 1$ and any $g \in K_l(GL_{n+k},F)$ there exist $p \in
P_l$ and $h\in H(F)$ such that $gx_0=
px_oh$. In other words, we have to solve the following equation:

$$
\begin{pmatrix}
  Id_k+A & B & Id_k+A+C\\
  D & Id_{n-k}+E & D+F\\
  G & H & Id_k+G+I
\end{pmatrix}=
\begin{pmatrix}
  Id_k & 0 & Id_k\\
  D' & Id_{n-k}+E' & D'+F'\\
  G' & H' & Id_k+G'+I'
\end{pmatrix}
\begin{pmatrix}
  Id_k+x & y & 0\\
  z & Id_k+w & 0\\
  0 & 0 & Id_k+h
\end{pmatrix},
$$
where all the \RGRCor{capital} letters denote matrices of appropriate sizes with entries in $\pi^l\cO_F$, and the matrices in the left hand side are parameters and matrices in the right hand side are unknowns.

The solution is given by:
\begin{multline*}
 x=A, \quad y=B, \quad z=D , \quad \RGRCor{w=} E , \quad  h= A+C \\
D' = 0, \quad E'=0, \quad F' = (D+F) (Id_k +A+C)^{-1},\\
H'=(H-G(Id_k+A)^{-1}B)(-D(Id_k+A)^{-1}B+Id_{n-k}+E)^{-1}\\
G'=(G-H'D)(Id_k+A)^{-1}, \quad
I'=(G+I-A-C)(Id_k+A+C)^{-1} - G'
\end{multline*}}

\end{proof}

\subsection{Structure of the spherical space $(\GL_{n+1}\times\GL_n)/\Delta \GL_n$} \lbl{subsec:PfGood}

Consider the embedding $\iota:\GL_n\hookrightarrow \GL_{n+1}$
given by
\[
A\mapsto\bmat 1&0\\0&A\emat.
\]

Denote $G=\GL_{n+1}(F)\times\GL_n(F)$ and $H=\Delta \GL_n(F)$.
The quotient space $G/H$ is isomorphic to $(\GL_{n+1})_R$ via the map
$(g,h)\mapsto g\iota(h^{-1})$. Under this isomorphism, the action
of $G$ on $G/H$ becomes $ (g,h)\cdot X=gX\iota(h^{-1})$.

The space $G/H$ is spherical.
Indeed, let $B\subset G$ be the
Borel subgroup consisting of pairs $(b_1,b_2)$, where $b_1$ is
lower triangular and $b_2$ is upper triangular, and let $x_0\in G/H$
be the point represented by the matrix
\[
x_0=\bmat1&e\\0&I\emat,
\]
where $e$ is a row vector of 1's. We claim that $Bx_0$ is open in
$G/H$. Let $\fb$ be the Lie algebra of $B$. It consists of pairs
$(X,Y)$ where $X$ is lower triangular and $Y$ is upper triangular.
The infinitesimal action of $\fb$ on $X$ at $x_0$ is given by
$(X,Y)\mapsto Xx_0-x_0d\iota(Y)$. To show that the image is
$\M_{n+1}$, it is enough to show that the images of the maps
$X\mapsto Xx_0$ and $Y\mapsto x_0d\iota(Y)$ have trivial
intersection. Suppose that $Xx_0=x_0d\iota(Y)$. Then
$X=x_0d\iota(Y)x_0^{-1}$, i.e.
\[
X=\bmat1&e\\&I\emat\bmat0&0\\0&Y\emat\bmat1&-e\\&I\emat=\bmat0&eY\\0&Y\emat.
\]
Since $X$ is lower triangular and $Y$ is upper triangular, both
have to be diagonal. But $eY=0$ implies that $Y=0$, and hence also
$X$=0. Proposition \ref{peop:F.sph} implies that the pair $(G,H)$ is $F$-spherical.

The following describes the quotient $G(O_F)\bs G(F)/H(F)$.

\begin{lem} \lbl{lem:singular.values} For every matrix $A\in\M_{n+1}(F)$ there are $k_1\in\GL_{n+1}(O)$ and $k_2\in\GL_n(O)$ such that
\begin{equation} \lbl{eq:pi.la}
k_1A\iota(k_2)=\bmat\pi^a&\pi^{b_1}&\pi^{b_2}&\dots&\pi^{b_n}\\&\pi^{c_1}&&&\\ &&\pi^{c_2}&\\&&&\ddots&\\ &&&&\pi^{c_n}\emat,
\end{equation}
where the numbers $a,b_i,c_i$ satisfy that if $i<j$ then $c_i-c_j\leq b_i-b_j\leq0$ and $b_1\leq c_1$.
\end{lem}

\begin{proof} Let $a$ be the minimal valuation of an element in the first column of $A$. There is an integral matrix $w_1$ such that the first column of the matrix $w_1A$ is $\pi^a,0,0,\ldots,0$. Let $C$ be the $n\times n$ lower-right sub-matrix of $w_1A$. By Cartan decomposition, there are integral matrices $w_2,w_3$ such that $w_2Cw_3^{-1}$ is diagonal, and its diagonal entries are $\pi^{c_i}$ for a non-decreasing sequence $c_i$. Finally, there are integral and diagonal matrices $d_1,d_2$ such that the matrix $d_1\iota(w_2)w_1A\iota(w_3^{-1})\iota(d_2^{-1})$  has the form (\ref{eq:pi.la}).

Suppose that $i<j$ and $b_i>b_j$. Then adding the $j$'th column to the $i$'th column and subtracting $\pi^{c_j-c_i}$ times the $i$'th row to the $j$'th row, we can change the matrix (\ref{eq:pi.la}) so that $b_i=b_j$. Similarly, if $i<j$ and $b_i-b_j<c_i-c_j$, then adding $\pi^{b_j-b_i-1}$ times the $i$'th column to the $j$'th column, and subtracting $\pi^{c_i+b_j-b_i-1-c_j}$ times the $j$'th row to the $i$'th row changes the matrix (\ref{eq:pi.la}) so that $b_i$ becomes smaller in $1$. Finally, if $c_1<b_1$ than adding the second row to the first changes the matrix so that $c_1=b_1$.
\end{proof}

Let $T\subset G$ be the torus
consisting of pairs $(t_1,t_2)$ such that $t_i$ are diagonal. The
co-character group of $T$ is the group $\bZ^{n+1}\times\bZ^n$. The
positive Weyl chamber of $T$ that is defined by $B$\footnote{The positive Weyl chamber defined by the Borel $B$ is the subset of co-weights $\la$ such that $\pi^{\la}B(O)\pi^{-\la}\subset B(O)$} is the set
$\De\subset X_*(T)$ consisting of pairs $(\mu,\nu)$ such that the
$\mu_i$'s are non-decreasing and the $\nu_i$'s are non-increasing.
Lemma \ref{lem:singular.values} implies that the set
$\left\{\pi^\la x_0\right\}_{\la\in\De}$ is a complete set of
orbit representatives for $G(O)\bs G(F)/H(F)$.\\

We are ready to prove that $(G,H)$ is uniform spherical.

\begin{prop} \lbl{prop:good.pair.GLn+1xGLn} The pair $((\GL_{n+1})_R\times(\GL_n)_R, \Delta (\GL_n)_R)$ is uniform spherical.
\end{prop}

\begin{proof}
Let $\Upsilon\subset X_*(T)$ be the positive Weyl chamber and let $\fX:=\{x_0\}$. By the above, the first condition of Definition \ref{defn:good.pair} holds. As for the second condition, an easy computation shows that if $a,b_1,\ldots,b_n,c_1,\ldots,c_n\in\bZ$, $a',b_1',\ldots,b_n',c_1',\ldots,c_n'\in\bZ$ satisfy the conclusion of Lemma \ref{lem:singular.values}, and $(k_1,k_2)\in G(O)$ satisfy that
\[
k_1\bmat\pi^a&\pi^{b_1}&\pi^{b_2}&\dots&\pi^{b_n}\\&\pi^{c_1}&&&\\ &&\pi^{c_2}&\\&&&\ddots&\\ &&&&\pi^{c_n}\emat\iota(k_1)=\bmat\pi^{a'}&\pi^{b_1'}&\pi^{b_2'}&\dots&\pi^{b_n'}\\&\pi^{c_1'}&&&\\ &&\pi^{c_2'}&\\&&&\ddots&\\ &&&&\pi^{c_n'}\emat,
\]
then $a=a'$, $c_i=c_i'$, $k_1$ has the form $\bmat1&B\\0&D\emat$,
where $B$ is a $1\times n$ matrix and $D$ is an $n\times n$ matrix
that satisfy the equations $D=\pi^c k_2\pi^{-c}$ and
$B\pi^c=\pi^b-\pi^{b'}k_2$, where $\pi^c$ denotes the diagonal
matrix with entries $\pi^{c_1},\ldots,\pi^{c_n}$, $\pi^b$ denotes
the row vector with entries $\pi^{b_i}$, and $\pi^{b'}$ denotes
the row vector with entries $\pi^{b_i'}$. The second condition of
Definition \ref{defn:good.pair} holds by Lemma \ref{lem:SmoothCrit}.

\RGRCor{The} third condition follows because, using the affine embedding as above,
$\pi^\la x_0$
 has the form (\ref{eq:pi.la}) and so
$val_F(\pi^\la x_0)$
 is independent of $F$.

\RGRCor{
Finally it is left to verify the last condition. In the following, we will distinguish between the $\ell$th congruence subgroup in $\GL_{n+1}(F)$, which we denote by $K_\ell(\GL_{n+1}(F))$, the $\ell$th congruence subgroup in $\GL_{n}(F)$, which we denote by $K_\ell(\GL_{n}(F))$, and the $\ell$th congruence subgroup in $G=\GL_{n+1}(F)\times\GL_n(F)$, which we denote by $K_\ell$. By lemma \ref{lem:GRCrit} it is enough to show that $(B\cap K_l) x_0=K_lx_0.$ It is easy to see that   $K_lx_0= x_0+\pi^{l}Mat_{n}(O_{F}).$ Let $y \in x_0+\pi^{l}Mat_{n}(O_{F})$. We have to show that $y\in (B\cap K_l) x_0$. In order to do this let us represent $y$ as a block matrix $$y=\bmat a&b\\c&D\emat,$$ where $a$ is a scalar and  $D$ is $n\times n$ matrix. Using left multiplication  by lower triangular matrix from $K_l(\GL_{n+1}(F))$  we may bring $y$ to the form $\bmat1&b'\\0&D'\emat$. We can decompose $D'=LU$, where $L,U \in K_l(\GL_{n+1}(F))$ and $L$ is lower triangular and $U$ is upper triangular. Therefore by action of an element from $B\cap K_l$ we may bring $y$ to the form $\bmat1&b''\\0&Id\emat$. Using right multiplication  by diagonal  matrix from $K_l(\GL_{n+1}(F))$  (with first entry 1) we may bring $y$ to the form $\bmat1&e\\0&D''\emat,$ where $e$ is a row vector of 1's and $D''$ is a diagonal matrix. Finally, using left multiplication  by diagonal  matrix from $K_l(\GL_{n+1}(F))$   we may bring $y$ to be $x_0.$  
}
\end{proof}

\subsection{The Pair $(\GL_{n+1}\times\GL_n,\Delta \GL_n)$}
\lbl{subsec:GL}$ $

In this section we prove Theorem \ref{thm:Mult1} which states that
$(\GL_{n+1}(F),\GL_n(F))$ is a strong Gelfand pair for any local
field $F$, i.e. for any irreducible smooth representations $\pi$
of $\GL_{n+1}(F)$ and $\tau$ of $\GL_{n}(F)$ we have
$$\dim \Hom_{\GL_{n}(F)} (\pi, \tau) \leq 1.$$
It is well known (see e.g. \cite[section 1]{AGRS})
that this theorem is equivalent to the statement that
$(\GL_{n+1}(F)\times\GL_n(F),\Delta \GL_n(F))$, where $\Delta
\GL_n$ is embedded in $\GL_{n+1}\times\GL_n$ by the map
$\iota\times Id$, is a Gelfand pair.

By Corollary \ref{cor:GelGel} this statement follows from Proposition \ref{prop:good.pair.GLn+1xGLn}, and the following
\RGRCor{theorem:}

\begin{thm}[\cite{AGRS}, Theorem 1]
Let $F$ be a local field of characteristic 0. Then
$(\GL_{n+1}(F),\GL_n(F))$ is a strong Gelfand pair.
\end{thm}

\end{document}